\UseRawInputEncoding

[1994/12/01]
\documentclass[manuscript]{amsart}
\usepackage{amsmath,amsthm,amssymb}

\usepackage{enumitem}

\usepackage{graphicx}

\usepackage{cite}

\makeatletter
\@namedef{subjclassname@2010}{%
  \textup{2010} Mathematics Subject Classification}
\makeatother

\usepackage[T1]{fontenc}

\newtheorem{theorem}{Theorem}[section]
\newtheorem{lemma}[theorem]{Lemma}
\newtheorem{corollary}[theorem]{Corollary}
\theoremstyle{definition}
\newtheorem{definition}[theorem]{Definition}

\theoremstyle{remark}

\newtheorem{proposition}[theorem]{Proposition}
\newtheorem{fact}[theorem]{Fact}
\numberwithin{equation}{section}

\begin{document}

\title[Weighted estimates of commutators     for $0<p<\infty$]{Weighted estimates of commutators     for $0<p<\infty$}


\author{Shunchao Long}
\address{School of Mathematics and Computational Science\\ Xiangtan University\\ Xiangtan 411105 China}
\curraddr{}
\email{sclongc@126.com}



\date{}

\dedicatory{}

\begin{abstract}
We establish weighted inequalities for $BMO$ commutators of  sublinear operators for all $0<p<\infty$.
\par For   weights $w$  satisfying the doubling condition of order $q$ with $0<q<p$ and the reverse H\"{o}lder condition, we prove that

\par  $\bullet$
    commutators $T_b$, which are bounded on  $L^p$  with  $1<p<\infty$, are bounded from some subspaces of $L^p_w$ to $L^p_w$  and to themselves for all $0<p<\infty$, these are applied to the commutators of singular integral operators and Hardy-Littelwood maximal operator, et.al, which are known to fail to be bounded from $H^1$ to $L^1$ and whose estimate has been open problems for $0<p$ enough small;

\par  $\bullet$
   commutators $T_b$, whose associated  operators $T$ are bounded on  $L^p$  with  $1<p<\infty$, are bounded from some subspaces of $L^p_w$ to $L^p_w$  and from some subspaces of $L^p_w$ to others  for all $0<p<\infty$, these are applied to the commutators of maximal operators such as singular integral maximal operators, Carleson  operator and the polynomial Carleson operator, et.al, the estimate of these commutators has been open problems for each $0<p<\infty$;

\par  $\bullet$
in particular,  these imply that the  commutators  above are bounded  from  $H^p_w$ to $L^p_w$  and to itself for all $0<p\leq 1$.

\end{abstract}

\maketitle

\par \section
 { Introduction }
Given a operator $T$ acting on functions and given
a function $b$, we define formally  the  commutator $T_b$ as
$$ T_bf(x)=T((b(x)-b(\cdot))f(\cdot))(x).$$
\par
A classical result of Coifman et al \cite{CRW} states that the commutator $T_b$
is bounded on $L^p$ for $ 1<p <\infty $,
 when $T$ is   a classical singular integral operator and $b \in  BMO$.
 This operator is more singular than the associated singular integral operator since  it fails, in
general, to be of both  weak type (1, 1)  and type   $(H^1,L^1)$, (see Perez   \cite{P} and  Paluszynski   \cite{Pal}, respectively).
Moreover, Harboure et al \cite{HST} showed that $T_b$ is bounded from $H^1$ to $L^1$ if
and only if $b$ is equal to a constant almost everywhere. And Perez
\cite{P} found a subspace $H^1_b $ of $H^1$ such that $T_b$ maps continuously
$H^1_b $ into $L^1$. Ky \cite{Ky} found the largest subspace of $H^1$
such that all commutators $T_b$ of Calder\'{o}n-Zygmund operators are bounded from this subspace into $L^1$.
\par
For a $\delta$-Calder\'{o}n-Zygmund operator $T$, whose distribution kernel $k(x,y)$ satisfies
$
|k(x, y)-k(x, z)| \leq C
\frac{|y-z|^{\delta}}{
|x - z|^{n+\delta} }
$
if $2 |y - z| < |x - z|$, for some $ 0<\delta \leq 1$,
 the above results  were extended to $p\leq 1$ large enough,  that is,  $T$ maps continuously $H^p$ into
$L^p$ for $\frac{n}{n+\delta} < p\leq 1$, however, the  commutator $T_b$ with $b\in BMO$ does
not map $H^p$ into $L^p$,   but map a subspace $H_b^p$ of $H^p$  into
$L^p$, for $\frac{n}{n+\delta} < p\leq 1$, see Alvarez \cite{Alv}. While for $0<p\leq \frac{n}{n+\delta} $, no boundedness estimate of $T_b$ is fond in the literatures.
\par
For the weighted case, Alvarez et al \cite{ABKP} proved that the
commutators $T_b$ of some linear operators $T$ with  $b$ in $BMO$ are bounded on the weighted Lebesgue spaces $L^p_w$ with $ 1<p<\infty$ and $w \in A_p$, where $A_p$ denotes the class of Muckenhoupt weights \cite{Muck}.
Similar to the unweighted case, $T_b$ may not be bounded from the weighted Hardy
space $H^1_w$ into the weighted Lebesgue space $L^1_w$ when  $b\in BMO$ and $w$ is a Muckenhoupt weight. Recently, Liang et al \cite{LKY} found out a  subspace  of $BMO$  such that, when $b$ belongs to this subspace, $T_b$ is bounded from $H^1_w$ to  $L^1_w$ for $w \in A_{1+\delta/n}$, where $T$ is  a $\delta$-Calder\'{o}n-Zygmund operator.
\par
In general, for Muckenhoupt class, when  $b\in BMO$, the commutator $T_b$ of a classical operator $T$ does not have  the same  endpoint estimate from $H^p_w$ into  $L^p_w$ as $T$.
Naturally, we ask whether  there exists   classes  of  weights such that, for these weights, when $b\in BMO$, the commutator $T_b$  has  the  same endpoint estimate from $H^p_w$ into  $L^p_w$ and  from $H^p_w$ into itself, as their associated   operator $T$ .
\par
At the same time,  only for part $p $ in $(0,1]$, the $BMO$ commutators have the endpoint estimates from a subspace of $H^p$ to $L^p$ or the weighted  endpoint estimates from a subspace of $H_w^p$ to $L_w^p$ for a  $w$ in $A_1$.
Naturally, we ask whether   there exists some endpoint estimates (or weighted endpoint estimates) of $BMO$ commutators  for all $p$ in $(0,1]$.
\par
  Although  the $L^p$  (or $L^p_w$) estimate with $1<p<\infty$ of the $BMO$ commutator of the classical singular integral operators were extended  to the commutators of many other operators, see \cite{SeTo, ABKP, GHST, CD},
  there are still many   operators, especially some maximal operators, such as Carleson  operator, et al,  the $L^p$ estimate  of their $BMO$ commutators  are  unknown.
Naturally, we ask  whether   there exists  estimates  of these commutators.
\par
The purpose of the present paper is to give  answers to the above questions.

 Let $  w\in A _{q,r}   $ with $ 0<q <p  $  and  $ 1<r <\infty  $, (see below for the definition).

 For the commutators, which are known to be bounded on $L^p$ for all $1<p<\infty$, which may also be known to be bounded from a subspace of  $H^p$ to $L^p$  for some  $0<p\leq 1$, we extend the estimate to all $0<p<\infty$,
  we prove that they are bounded  from some subspaces of $L^p_w$ to $L^p_w$  and to themselves for all $0<p<\infty$, in particular,  we get  that they are bounded   from  $H^p_w$ to $L^p_w$  and to itself for all $0<p\leq 1$.

  For the commutators $T_b$ which are still unknown whether they are bounded on $L^p$ for  $1<p<\infty$ but the associated operators $T$  is known to be bounded  on $L^p$ for all   $1<p<\infty$, we give a estimate for all $0<p<\infty$, we prove that they  are bounded from some subspaces of $L^p_w$ to $L^p_w$  and from some subspaces of $L^p_w$ to others  for all $0<p<\infty$, in particular,  we get also that they are bounded   from  $H^p_w$ to $L^p_w$  and to itself for all $0<p\leq 1$.

\par
\par
Let us first intruduce some definitions.

\par  A  nonnegative local integrable function $w$ is called a weight. $w(Q)$ denotes $\int_Qw$.  For any cube $Q$ and $\lambda>0$,  $ \lambda Q $ denotes the cube concentric with $Q$ whose each edge is $\lambda$ times as long. We write  $\bar{p}=\inf \{p,1\}$ for $0<p<\infty$.

\begin{definition}
  Let   $0< p<\infty$. We say  a weight $w\in D_{p}$  if    $w$  satisfies    the doubling condition of order $p$
 \begin{equation}\label{2.2}
 w(\lambda Q)\leq C  \lambda^{np} w(Q)
  \end{equation}
   for any cube $Q $ and $\lambda>1$, where $C$ is a  constant  independent of $Q $ and $\lambda$.
\end{definition}

\par
\begin{definition}
  Let   $1< r<\infty$. We say a weight
 $w \in RH_r$  if $w$
 satisfies   the reverse H\"{o}lder condition of order $r$
$$ \left(\frac{1}{|Q|}\int_Q w^{r}\right)^{1/r}\leq \frac{C}{|Q|}\int_Q w$$
 for  every cube $Q$,  where  $C$ is a  constant  independent of $Q$.

\end{definition}

\par We denote $D _{p} \cap RH_r $ by  $A _{p,r} $ for $ 0<p <\infty  $  and  $ 1<r <\infty  $.

\par
\begin{definition}\label{def_2.1}
 We say a weight $w\in P $ if there exist a sequence $\{Q_i\}$ of cubes with $w(Q_i)>0$ for each $i$ and whose interiors are disjoint each other such that  ${\bf R}^n=\bigcup_{i=1}^{\infty} Q_i$ .
\end{definition}
\par

\par
\begin{definition}\label{def_3.1}
Let $  0< p < \infty ,
0< s \leq \infty$ and $ w$ be a weight.
 A function $h$ is said to be a
 $ (p, s, w)$-block, if there is a cube  $Q \subset {\bf R}^n$ for which
\par (i)~~~~supp $ h\subseteq Q$,
\par (ii)~~~~$\|h\|_{L^{s}   }\leq |Q|^{1/s}w(Q)^{-1/p}.$
\end{definition}
\par
\begin{definition}\label{def_3.1}
 Let $ 0< p <  \infty, w\in A _{q,r}   $ with $ 0<q <p  $  and  $ 1<r <\infty  $,
and  $  1 < s\leq \infty  $ and  $  rp/(r-1) \leq s  $.
\par
The spaces $BH^{p,s}_{w}$ consists of functions $f$ which can be written  as
$$f=\sum\limits_{k=1}^{\infty} \lambda _k h_k ,$$
   where
  $h_k$ are $(p,s, w)$-blocks and $\lambda _k$ are real numbers with
 $\sum\limits_{k=1}^{\infty} |\lambda _k|^{\bar{p}} <+ \infty .$
\par
 We  equip $BH^{p,s}_{w}$ with the quasi-norm,
$$
\|f\|_{BH^{p,s}_{w}}=
 \inf  \left( \sum\limits_{k=1}^{\infty} |\lambda _k|^{\bar{p}} \right)^{1/\bar{p}},
  $$
where the infimum is taken over all the above decompositions of $f$.
\end{definition}

\par We have the following facts:
\par
\begin{fact}\label{Re3.1}
Under the conditions of $p,w,q,s$ in Definition 1.5,
every element  $\sum_{i=1}^{\infty}\lambda_i a_i$ in $BH^{p,s}_{w} $  converges in $ H^{p}_{w} $ and $w$-a.e..
\end{fact}
\par
\begin{fact}\label{Re3.1}
 Let $ 0< p \leq 1$,
  $ w\in A _{q,r}  $ with $ 0<q <p  $  and  $ 1<r <\infty  $,
and  $   1 < s\leq \infty  $ and $  rp/(r-1) \leq s   $, we have
 $BH^{p,s}_{w} = H^{p}_{w}$, the classical weighted Hardy spaces.
\end{fact}
\par
See \cite{L} for above definitions and facts. $w$-a.e. convergence means convergence for all $x\in E^c$ for some $E$ with $w(E)=0$.
\par
Recall that a locally integrable function $b$ is said to be in $BMO$ if
$$\|b\|_{BMO} := \sup_{Q}\frac{1}{|Q|}\int_{Q}|b(x) - b_Q| dx<\infty,$$
where $b_Q=\frac{1}{|Q|}\int_{Q}b(x)  dx$,
the supremum is taken over all cubes $Q \subset {\bf R}^n$.

\par
We will consider the  operator $T$ defined  for every $(p,s,w)-$bolck $h$ with supp $h \subset Q$, a cube with the cental  $x_0$, and  satisfying the size condition
\begin{equation}\label{1.1}
|Th(x)| \leq  C \frac{  \|h\|_{L^{1}}}{|x-x_0|^n}
\end{equation}
for $x\in (2n^{1/2}Q)^c $,
 where $C$ is a constant independent of $h$. $T$  is either a linear operator or a sublinear operator whose commutator satisfies the following condition
\begin{equation}\label{7.2}
|T_b(\sum\lambda_ja_j)(x)|\leq \sum|\lambda_j||T_ba_j(x)|,~w{\rm -a.e.},
\end{equation}
 for $f=\sum\lambda_ja_j \in   {B}L^{p,s}_{w}$.
  \par Here and below, $f=\sum\lambda_ja_j \in   {B}L^{p,s}_{w}$ (or $  H^{p}_{w}$) means that    each  $a_j$ is a $(  p,s, w)$-block and $\sum|\lambda_j|^{\bar{p}}<\infty$.

\par We will state our theorems and their applications in Section 2, and prove the theorems in Section 3.

\par Throughout the whole paper,  $C$ denotes a positive absolute constant not necessarily the same at each occurrence.
$s'$ denotes the conjugate number of $s$ with $1\leq s\leq \infty$, which satisfies $1/s+1/s'=1$.

We express our gratitude   to   David Cruz-Uribe   for his comments
which led to substantial  improvements of this paper.

\section
 { Theorems  and their applications}

\subsection{Theorems}
 Now let us state our theorems for   $BMO$ commutators of some sublinear operators.
\par
 \subsubsection{ }  For the commutators  $T_b$ which are known to be  bounded on $L^p$ with $1<p<\infty$,    we have the following theorems.
\par
\begin{theorem}\label{th_1}
    Let $0<p <\infty ,  w\in A _{q,r}  $ with $ 0<q <p  $  and  $ 1<r <\infty  $,   $  1   < s \leq \infty $ and $   rp/(r-1)  \leq s  $. Let $b\in BMO$.
      Suppose that a operator $T$  is defined for every $(p,s,w)$-block $h$  and   satisfies (1.2).  Suppose that
    $ \|T_b h\|_{L^{s}}\leq C \|b\|_{BMO} \| h\|_{L^{s}} $.

\par  (i)  If  $ T $  is a linear operator, then
  $T_b$ has an unique bounded extension (still denoted by $T_b$) from $BH^{p,s}_{w}$   to $L^{p}_{w}$ that satisfies
 \begin{equation}\label{7.4}
  T_b(\sum_{j=1}^{\infty}\lambda_ja_j)(x) =\sum_{j=1}^{\infty}\lambda_jT_ba_j(x),
  \end{equation}
  in $L^p_w$ and $w$-a.e., and
  \begin{equation}\label{7.4}
  \|T_bf\|_{L^{p}_{w}}   \leq   C \|b\|_{BMO} \|f\|_{BH^{p,s}_{w}},
    \end{equation}
    for all  $ f= \sum_{j=1}^{\infty}\lambda_ja_j \in BH^{p,s}_{w}$.

   \par (ii)  If $ T $ is a sublinear operator whose commutator satisfies (1.3)  for    $f=\sum\lambda_ja_j \in BH^{p,s}_{w} $,
 then $T_b$ has a bounded extension  from $BH^{p,s}_{w}$    to $L^{p}_{w}$  satisfying  (2.2).
  \end{theorem}

\par  In particularly, we have from Theorem 2.1 and Fact 1.7 that

   \begin{corollary}
    Let $0<p \leq 1 $ and      $    w\in A _{q,r}  $ with $ 0<q <p  $  and  $ 1<r <\infty  $. Let $b\in BMO$.
     Suppose that a operator $T$  is defined for every $(p,s,w)$-block $h$ with  $  1   < s \leq \infty $ and $   rp/(r-1)  \leq s  $,  and satisfies (1.2).
    Suppose that     $ \|T_b h\|_{L^{s}}\leq C \|b\|_{BMO} \| h\|_{L^{s}} $.

\par     (i)  If  $ T $  is a linear operator, then
  $T_b$ has an unique bounded extension (still denoted by $T_b$) from $H^{p}_{w}$
    to $L^{p}_{w}$ that satisfies (2.1)   in $L^p_w$ and $w$-a.e., and
    \begin{equation}\label{7.4} \|T_b f\|_{L_w^{p}}\leq C \|b\|_{BMO} \| f\|_{H^{p}_{w}} \end{equation}
     for all  $ f= \sum_{j=1}^{\infty}\lambda_ja_j \in H^{p}_{w}(=BH^{p,s}_{w} )$.

   \par (ii) If $ T $ is a sublinear operator whose commutator satisfies (1.3)  for    $f=\sum\lambda_ja_j \in H^{p}_{w}$,
 then $T_b$ has a bounded extension  from $H^{p}_{w}$    to $L^{p}_{w}$  that satisfies  (2.3).
   \end{corollary}

  \begin{theorem}\label{th_2}
   Let $0<p <\infty ,  w\in A _{q,r}  $ with $ 0<q <p  $  and  $ 1<r <\infty  $,  and $\max\{ p/q, rp/(r-1) \} < s < \infty $. Let $b\in BMO$.
      Suppose that a operator $T$  is defined for every $(p,s,w)$-block $h$  and   satisfies (1.2).  Suppose that
    $ \|T_b h\|_{L^{s}}\leq C \|b\|_{BMO} \| h\|_{L^{s}} $.

 \par (i)  If  $ T $  is a linear operator,
  then
  $T_b$ has an unique bounded extension (still denoted by $T_b$) from $BH^{p,s}_{w}$     to itself that satisfies (2.1)
   in $BH^{p,s}_{w}$ and $w$-a.e., and
  \begin{equation}\label{7.4}
  \|T_bf\|_{BH^{p,s}_{w}}   \leq   C \|b\|_{BMO} \|f\|_{BH^{p,s}_{w}},
    \end{equation}
    for all  $ f= \sum_{j=1}^{\infty}\lambda_ja_j \in BH^{p,s}_{w}$.

   \par (ii)  If $ T $ is a sublinear operator whose commutator satisfies (1.3)  for    $f=\sum\lambda_ja_j \in BH^{p,s}_{w} $,  and $w \in P$,
 then $T_b$ has a bounded extension  from $BH^{p,s}_{w}$    to itself  that satisfies  (2.4).
   \end{theorem}

\par   In particularly, we have from Theorem 2.3 and Fact 1.7 that
   \begin{corollary}\label{co_2}
     Let $0<p \leq 1 $ and      $ w\in A _{q,r}  $ with $ 0<q <p  $  and  $ 1<r <\infty  $. Let $b\in BMO$.
     Suppose that a operator $T$  is defined for every $(p,s,w)$-block $h$ with $\max\{ rp/(r-1),p/q\}< s < \infty  $ and satisfies (1.2).
    Suppose that     $ \|T_b h\|_{L^{s}}\leq C \|b\|_{BMO} \| h\|_{L^{s}} $.

  \par   (i)  If $ T $  is a linear operator,  then
  $T_b$ has an unique bounded extension (still denoted by $T_b$) from $H^{p}_{w}$
    to itself that satisfies (2.1)   in $H^p_w$ and $w$-a.e., and
    \begin{equation}\label{7.4} \|T_b f\|_{H_w^{p}}\leq C \|b\|_{BMO} \| f\|_{H^{p}_{w}} \end{equation}
     for all  $ f= \sum_{j=1}^{\infty}\lambda_ja_j \in H^{p}_{w}(=BH^{p,s}_{w} )$.

   \par (ii)  If $ T $ is a sublinear operator whose commutator satisfies (1.3)  for    $f=\sum\lambda_ja_j \in H^{p}_{w} $. and $w\in  P$,
 then $T_b$ has a bounded extension  from $H^{p}_{w}$    to itself that satisfies  (2.5).
   \end{corollary}

 \subsubsection{ } For the commutators $T_b$ that are not yet known whether they  are bounded on $L^p$, but   knew that the associated operators $T$  are bounded  on $L^p$,   $1<p<\infty$,  we have the following theorems.

   \begin{theorem}\label{th_3}
    Let $0<p <\infty ,  w\in A _{q,r}  $ with $ 0<q <p  $  and  $ 1<r <\infty  $,  and $\max\{ 1, rp/(r-1) \} < s \leq \infty $. Let $b\in BMO$.
      Suppose that a operator $T$  is defined for every $(p,s,w)$-block $h$  and   satisfies (1.2).  Suppose that
     $$ \|T h\|_{L^{s}}\leq C \| h\|_{L^{s}}, $$
    and there is a $\tilde{s}$, which satisfies $\max\{ rp/(r-1),s/(s+1) \}\leq \tilde{s}  <s $,  such that
          $$ \|T h\|_{L^{\tilde{s}}}\leq C \| h\|_{L^{\tilde{s}}}. $$

 \par (i)  If  $ T $  is a linear operator, then
  $T_b$ has an unique bounded extension (still denoted by $T_b$) from $BH^{p,s}_{w}$
    to $L^{p}_{w}$ that satisfies (2.1)
   in $L^p_w$ and $w$-a.e., and (2.2)
      for all  $ f= \sum_{j=1}^{\infty}\lambda_ja_j \in BH^{p,s}_{w}$.

   \par (ii)  If $ T $ is a sublinear operator satisfying (1.3)  for    $f=\sum\lambda_ja_j \in BH^{p,s}_{w} $,
  then  $T_b$ has a bounded extension  from $BH^{p,s}_{w}$    to $L^{p}_{w}$  that satisfies  (2.2).
   \end{theorem}

  \par In particularly, we have from Theorem 2.5 and Fact 1.7 that

    \begin{corollary}\label{co_3}
     Let $0<p \leq 1 $ and $   w\in A _{q,r}  $ with $ 0<q <p  $  and  $ 1<r <\infty  $.   Let $b\in BMO$.
      Suppose that a operator $T$  is defined for every $(p,s,w)$-block $h$ with $\max\{ 1, rp/(r-1) \} < s \leq \infty $ and   satisfies (1.2).  Suppose that
     $$ \|T h\|_{L^{s}}\leq C \| h\|_{L^{s}}, $$
    and there is a $\tilde{s}$, which satisfies $\max\{ rp/(r-1),s/(s+1) \}\leq \tilde{s}  <s $,  such that
          $$ \|T h\|_{L^{\tilde{s}}}\leq C \| h\|_{L^{\tilde{s}}}. $$

\par (i)  If  $ T $  is a linear operator, then
  $T_b$ has an unique bounded extension (still denoted by $T_b$) from $H^{p}_{w}$     to $L^{p}_{w}$ that satisfies (2.1)    in $H^p_w$ and $w$-a.e., and (2.3)
      for all  $ f= \sum_{j=1}^{\infty}\lambda_ja_j \in H^{p}_{w} (=BH^{p,s}_{w})$.

   \par (ii)  If $ T $ is a sublinear operator whose commutator satisfies (1.3)  for    $f=\sum\lambda_ja_j \in H^{p}_{w} $,
  then $T_b$ has a bounded extension  from $H^{p}_{w}$    to $L^{p}_{w}$ that  satisfies  (2.3).
    \end{corollary}

  \begin{theorem}\label{th.2.7}
    Let $0<p <\infty $ and      $  w\in A _{q,r}  $ with $ 0<q <p  $  and  $ 1<r <\infty  $. Let $\max\{ p/q, rp/(r-1) \} < \tilde{s} < s  \leq \infty $ and $s/(s+1) \leq \tilde{s}  <s $.
    Let $b\in BMO$.
      Suppose that a operator $T$  is defined for every $(p,s,w)$-block $h$  and   satisfies (1.2).  Suppose that
             $ \|T h\|_{L^{ s}}\leq C \| h\|_{L^{s}}, $
          and
          $ \|T h\|_{L^{\tilde{s}}}\leq C \| h\|_{L^{\tilde{s}}}. $

  \par     (i) If  $ T $  is a linear operator, then
  $T_b$ has an unique bounded extension (still denoted by $T_b$) from $BH^{p,s}_{w}$
    to $BH^{p, \tilde{s}}_{w}$ that satisfies (2.1)
   in $BH^{p, \tilde{s}}_{w}$ and $w$-a.e., and
    \begin{equation}
     \|T_b f\|_{BH^{p,\tilde{s}}_{w}}\leq C \|b\|_{BMO} \| f\|_{BH^{p,s}_{w}}
      \end{equation}
    for all $f=\sum\lambda_ja_j\in BH^{p,s}_{w}$.

   \par (ii) If $ T $ is a sublinear operator  whose commutator satisfies (1.3)  for    $f=\sum\lambda_ja_j \in BH^{p,s}_{w} $, and  $w\in  P$,
 then $T_b$ has a bounded extension  from $BH^{p,s}_{w}$     to $BH^{p, \tilde{s}}_{w}$ that satisfies  (2.6).
   \end{theorem}

\par   In particularly, we have from Theorem 2.7 and Fact 1.7 that

   \begin{corollary}\label{co.4}
    Let $0<p \leq 1 $ and      $ w\in A _{q,r}  $ with $ 0<q <p  $  and  $ 1<r <\infty  $. Let $b\in BMO$.
   Suppose that  $T$ is defined for every $(p,s,w)$-block $h$ with $\max\{ p/q, rp/(r-1) \} < s \leq \infty $ and
       satisfies (1.2).  Suppose that
        $$ \|T h\|_{L^{s}}\leq C \| h\|_{L^{s}}, $$
        and there is a $\tilde{s}$, which satisfies  $\max\{ p/q, rp/(r-1) \} < \tilde{s} < s  \leq \infty $ and $s/(s+1) \leq \tilde{s}  <s $, such that
      $$ \|T h\|_{L^{ \tilde{s}}}\leq C \| h\|_{L^{\tilde{s}}}. $$

 \par    (i)  If  $ T $  is a linear operator, then
  $T_b$ has an unique bounded extension (still denoted by $T_b$) from $H^{p}_{w}$
    to itself that satisfies (2.1)   in $H^p_w$ and $w$-a.e., and (2.5)          for all  $ f= \sum_{j=1}^{\infty}\lambda_ja_j \in H^{p}_{w}  $.

   \par (ii) If $ T $ is a sublinear operator  whose commutator satisfies (1.3)  for    $f=\sum\lambda_ja_j \in H^{p}_{w}  $, and $w\in P$,
 then $T_b$ has a bounded extension  from $H^{p}_{w}$    to itself  that satisfies  (2.5).
   \end{corollary}

\par \subsection  { Applications}
Let  $T$ be a operator satisfying the size condition
\begin{equation}\label{7.11}
|Tf(x)|\leq C \int_{{\bf R}^n}\frac{|f(y)|}{|x-y|^n}dy, ~~~~ x\notin
{\rm supp}f,
\end{equation}
for $f$ with compact support set.

It is easy to check that  (2.7) implies (1.2), (see also Lemma 7.1 in \cite{L}).
Therefore, if condition (1.2) replaced by condition (2.7),  Theorem 2.1,2.3,2.5 and 2.7 hold.
While it is easy to see the condition (2.7) is satisfied by many classical operators in harmonic analysis, (see also  \cite{SW}).
 Thus, Theorem 2.1,2.3,2.5 and 2.7 are applied to these operators.

 \par  \subsubsection{ }  For the commutators that are known to be bounded  on $L^p$ with $1<p<\infty$, we can apply Theorem 2.1 and Theorem 2.3.

   \par These commutators include the following.

\par $\bullet$ {\bf The commutator of Hilbert transform $H$}
  $$H_bf(x)= {\rm p.v.~}\int_{{\bf R}}\frac{(b(x)-b(y))f(y)}{x-y}dy. $$
The operator $H$ is defined if we replace $b(x)-b(y)$ by 1 in the above definition of the commutator $H_b$. The following operators such as $R^j$, et al, are defined by the same way.

$H_b$ is  bounded on $L^p$ and $L^p_w$ with  $1<p<\infty$ and $w\in A_p$, (see \cite{Blo}). But $H_b$ fails to be   bounded from $L^1$ to weak-$L^1$, (see \cite{Pal}).  $H_b$ also fails to be   bounded from $H^1$ to $L^1$, (see \cite{P}).
While $H_b$ is  bounded from a subspace $H_b^p$ of $H^p$  to $L^p$ for $\frac{1}{2} <p\leq 1$, (the proof is similar to the case of $p=1$  in \cite{P}, see also \cite{Alv}).

  \par $\bullet$ {\bf The commutator of Riesz transform $R^j$}
$$R^j_b={\rm p.v.~}\int_{{\bf R}^n}  \frac{|x_j-y_j|}{|x-y|^{n+1}} (b(x)-b(y)) f(y)dy , j=1,2,\cdots,n.$$

$R^j_b$ is  bounded on $L^p$ and $L^p_w$ with  $1<p<\infty$ and $w\in A_p$, (see \cite{SeTo}).  $R^j_b$ is also bounded from  $H_b^p$   to $L^p$ for $\frac{n}{n+1} <p\leq 1$, (the proof is similar to the case of $p=1$  in \cite{P}, see also \cite{Alv}).

  \par $\bullet$ {\bf The commutator of the rough singular integral operators $T^{\Omega}$}
   $$T^{\Omega}_b f(x)={\rm p.v.~}\int_{{\bf R}^n} \frac{\Omega(\frac{x-y}{|x-y|})}{|x-y|^n} (b(x)-b(y))  f(y)dy, $$
where $ \Omega \in L^\infty (S^{n-1})$ and $\int _{{\bf S}^{n-1}}\Omega(u)d\sigma(u)=0$.

 $T^{\Omega}_b $ is  bounded on $L^p$ and $L^p_w$ with  $1<p<\infty$ and $w\in A_p$, (see \cite{ABKP}).

\par $\bullet$ {\bf The commutator of the C.Fefferman type strongly singular multiplier operator $T^F$}
$$T^F_b f(x)={\rm p.v.~} \\ \int_{{\bf R}^n}\frac{e^{i|x-y|^{-\lambda}}}{|x-y|^n} \chi_E(|x-y|) (b(x)-b(y))  f(y)dy, $$
where $0<\lambda<\infty$ and $\chi_E$ is the characteristic function of the unit interval $E=(0,1)\subset {\bf R} $.

  $T^F_b $ is  bounded on $L^p$ and $L^p_w$ with  $1<p<\infty$ and $w\in A_p$, (see \cite{GHST}).

\par $\bullet$  {\bf The commutator of Calder\'{o}n-Zygmund  operator $T^{CZ}$}

Let  $T^{CZ} $ is a Calder\'{o}n-Zygmund operator satisfying
 $$
T^{CZ}f(x) = \int k(x,y)f(y)dy,~~ x\in supp (f),
$$
 for any $f\in L^p$ with compact support for some $1<p<\infty$, where $k(x,y)$ is a standard Calder\'{o}n-Zygmund kernel   which satisfies
$$
|k(x, y)| \leq C
\frac{1}{
|x - y|^{n} },
$$
$$
|k(x, y)-k(x, z)| +|k(y, x)-k(z, x)| \leq C
\frac{|y-z|^{\delta}}{
|x - z|^{n+\delta} }
$$
for  $2 |y - z| < |x - y|$ and some $ 0<\delta \leq 1$. We define the  commutator of Calder\'{o}n-Zygmund  operator
$$
T^{CZ}_b f(x) = \int k(x,y)(b(x)-b(y))f(y)dy,~~ x\in supp (f),
$$
 for any $f\in L^p$ with compact support for some $1<p<\infty$.

$T^{CZ}_b $ is  bounded on $L^p$ and $L^p_w$ with  $1<p<\infty$ and $w\in A_p$, (see  \cite{ABKP} and \cite{SeTo}).
 $T^{CZ}_b$ is also bounded from a subspace $H_b^p$ of $H^p$  to $L^p$ for $\frac{n}{n+\delta} <p\leq 1$,  see  \cite{Alv}.

\par $\bullet$ {\bf The commutator of the partial sun operator $C^{\xi}$ of Fourier series}
  $$ C^{\xi}_bf(x)=\int_{{\bf R} }\frac{e^{2\pi i \xi y}}{x-y} (b(x)-b(y))  f(y)dy, $$
 where $\xi \in {\bf R}$.

 $C_b$ is  bounded on $L^p$ and $L^p_w$ with  $1<p<\infty$ and $w\in A_p$, (see \cite{SeTo}).

\par $\bullet$ {\bf  The commutator of Bochner-Riesz means    $B^{(n-1)/2,R}$ at the  critical index}
$$
B^{(n-1)/2,R}_b f(x)=\int_{{\bf R}^n}  K_R^{(n-1)/2}(x-y)(b(x)-b(y))  f(y)dy, 1<R<\infty,
$$
where $K^{(n-1)/2}_R(x) =
[(1-|\xi/R|^2)^{(n-1)/2}_+]~\breve{}~(x),$ and $\check{g}$ denotes the
inverse Fourier transform of $g$.

  Combining the results of  Alvarez  et al  in \cite{ABKP} and   Shi et al   in \cite{SS}, we know that $ B_b^{(n-1)/2,R} $ is bounded on $L^p$ with $1<p<\infty$.

\par $\bullet$ {\bf The commutators of  pseudo-differential operator whose symbols in $ {\mathcal{S}}^{n(\varrho-1)}_{\varrho,\delta}$ with $0<\varrho\leq 1, 0\leq \delta<1$.}

 Let $m\in {\bf R},0<\varrho\leq 1, 0\leq \delta<1$.
A symbol in $ {\mathcal{S}}^{m}_{\varrho,\delta}$ will be a smooth function $p(x, \xi)$ defined on ${\bf R}^n\times {\bf R}^n$,  satisfying the estimates
$$
D^{\alpha}_xD^{\beta}_{\xi}p(x, \xi)\leq C_{\alpha,\beta}(1+|\xi|)^{m-\varrho |\beta| +\delta|\alpha|}.
$$
As usual, $\mathcal{L}^{m}_{\varrho,\delta}$
will denote the class of operators with symbol in $ {\mathcal{S}}^{m}_{\varrho,\delta}$.
 Let $L\in \mathcal{L}^{m}_{\varrho,\delta}$.
 Following \cite {AH},  $L$ has a distribution kernel $k(x, y)$  defined by
the oscillatory integral
 $$
 k(x,y)=(2\pi)^{-n} \int e^{i(x-y)\cdot \xi} p(x, \xi)d\xi,
 $$
which satisfies, if $m+M+n>0$  for some $M\in {\bf Z}_+$, that
\begin{equation}\label{2.1}
D^{\alpha}_x D^{\beta}_{y}k(x,y )\leq C|x-y|^{-(m+M+n)/\varrho },~x\neq y.
\end{equation}
From (2.3), it is easy to see that
the  pseudo-differential operator $L$ satisfy (2.1) when $M=0$ and $ m=n(\varrho-1)$.  At the same time, $1<p<\infty, 0<\varrho\leq 1$ and $ 0\leq \delta<1$ imply $n(\varrho-1)\leq n(\varrho-1)|\frac{1}{p}-\frac{1}{2}|+ \min \{0,\frac{n(\rho-\delta)}{2}\}$, by Theorem 3.4 in \cite {AH}, we knew that
$L \in  {\mathcal{L}}^{n(\varrho-1)}_{\varrho,\delta}$ is bounded on $L^p$  with $1<p<\infty$. Let $L \in  {\mathcal{L}}^{n(\varrho-1)}_{\varrho,\delta}$,  we define the commutator
$$
L_b f(x) = \int k(x,y) (b(x)-b(y))  f(y)dy
$$
for $f\in C_0^{\infty}$.

  $L_b $ is  bounded on $L^p$ and $L^p_w$ with  $1<p<\infty$ and $w\in A_p$, (see \cite{ABKP}).

\par Denote  $T_b^{1}$  the above commutators $H_b, R^j_b,T^{\Omega}_b$, $T^F_b,  T^{CZ}_b, C^{\xi}_b, B^{(n-1)/2,R}_b$ and $L_b$.

 As we see from above, each  $T_b^{1}$ is bounded  on $L^p$ for all  $1<p<\infty$. But $T_b^{1}$ may fail to be   bounded from $L^1$ to weak-$L^1$.  $T_b^{1}$ may also fail to be   bounded from $H^1$ to $L^1$.
 Some of $T_b^{1}$  are bounded
   from a subspace of $H^p$ to $L^p$ for some $0<p\leq 1$. But, it seems that
   every $T_b^{1}$ has still no norm estimate for $0<p$ small  enough.

 We have from Theorem 2.1(i) and Theorem 2.3(i) that
 \begin{theorem}\label{th_5}
Let $0<p <\infty ,   w\in A _{q,r}  $ with $ 0<q <p  $  and  $ 1<r <\infty  $, and $b\in BMO$.

\par (i) If   $  1   < s \leq \infty $ and $   rp/(r-1)  \leq s  $,
 then each $T_b^{1}$  extends to a  bounded operator from $BH^{p,s}_{w}$ to $L^{p}_{w}$  that satisfies (2.2).

\par  (ii) If  $\max\{ p/q, rp/(r-1) \} < s < \infty $,
 then each $T_b^{1}$    extends to  a  bounded operator from $BH^{p,s}_{w}$ to itself  that satisfies (2.4).
 \end{theorem}

 \par In particularly, noticing Fact 1.7,  we have
\begin{corollary}\label{th 5}
Let $0<p \leq 1 ,$ and $  w\in A _{q,r}  $ with $ 0<q <p  $  and  $ 1<r <\infty  $, and $b\in BMO$.
Then,
\par (i) each $T_b^{1}$     extends  to  a  bounded operator from $H^{p}_{w}$ to $L^{p}_{w}$  that satisfies (2.3),
 \par (ii) each $T_b^{1}$     extends  to a  bounded operator  from $H^{p}_{w}$ to itself  that satisfies (2.5).
\end{corollary}

\par $\bullet$ {\bf The commutator of Hardy-Littlewood maximal operator $M$}
$$M_bf(x)= \sup _{x\in Q} \frac{1}{|Q|}\int_{Q}|b(x)-b(y)| |f(y)|dy.$$

$M_b$ is bounded on $L^p$ and $L^p_w$ with  $1<p<\infty$ and $w\in A_p$, (see   \cite{GHST}).   $M$ fails to be  bounded on $H^p$ fot all $0<p\leq 1$, since the  vanishing properties of $H^p$ functions,  (see \cite{S3},p129).

 We have from Theorem 2.1(ii) and Theorem 2.3(ii) that
 \begin{theorem}\label{th_5}
Let $0<p <\infty ,   w\in A _{q,r}  \cap  P$ with $ 0<q <p  $  and  $ 1<r <\infty  $,   and $b\in BMO$.

\par (i) If   $  1   < s \leq \infty $ and $   rp/(r-1)  \leq s  $,
 then   $M$  extends to a  bounded operator from $BH^{p,s}_{w}$ to $L^{p}_{w}$  that satisfies (2.2).

\par  (ii) If  $\max\{ p/q, rp/(r-1) \} < s < \infty $,
 then each $M$    extends to  a  bounded operator from $BH^{p,s}_{w}$ to itself  that satisfies (2.4).
 \end{theorem}

 \par In particularly, noticing Fact 1.7,  we have
\begin{corollary}\label{th 5}
Let $0<p \leq 1 ,$ and $  w\in A _{q,r}  \cap  P$ with $ 0<q <p  $  and  $ 1<r <\infty  $,  and $b\in BMO$.
Then,
\par (i)    $M$     extends  to  a  bounded operator from $H^{p}_{w}$ to $L^{p}_{w}$  that satisfies (2.3),
\par (ii)    $M$     extends  to  a  bounded operator  from $H^{p}_{w}$ to itself  that satisfies (2.5).
\end{corollary}

\par $\bullet$ {\bf The commutator of the maximal operator of  rough singular integral $\tilde{T}^{\Omega, *}$}
$$\tilde{T}^{\Omega ,*}_b f(x)=\sup_{j\in {\bf Z}} \int_{|x-y|>2^j} \frac{\Omega(\frac{x-y}{|x-y|})}{|x-y|^n}  (b(x)-b(y))  f(y)dy $$
 with $ \Omega \in L^\infty (S^{n-1}),\int _{{\bf S}^{n-1}}\Omega(u)d\sigma(u)=0$,
 and
 $$\sup _{\xi\in {\bf S}^{n-1} }
  \int _{{\bf S}^{n-1}}|\Omega(u)|\left( {\rm log }\frac{1}{|\xi\cdot u|} \right)^{1+\alpha}d\sigma(u)$$
for all $\alpha>2$.

$ \tilde{T}^{\Omega, *}_b $ is bounded on $L^p$ with $1<p<\infty$, (see  \cite{CD}). $ \tilde{T}^{\Omega, *}_b $ fails to be  bounded on $H^p$ fot all $0<p\leq 1$, since the  vanishing properties of $H^p$ functions,  (see \cite{S3},p129).

We have from Theorem 2.1(ii) and Theorem 2.3(ii) that

\begin{theorem}\label{th 6}
Let $0<p <\infty ,  w\in A _{q,r}   $ with $ 0<q <p  $  and  $ 1<r <\infty  $,  and $b\in BMO$.

\par (i) If    $  1   < s \leq \infty $ and $   rp/(r-1)  \leq s  $, then  $ \tilde{T}^{\Omega, *}_b $   extends  to a  bounded operator from $BH^{p,s}_{w}$ to $L^{p}_{w}$ that satisfies (2.2).

\par (ii)   If $ w\in  P$ and  $\max\{ p/q, rp/(r-1) \} < s < \infty $,    then
   $ \tilde{T}^{\Omega, *}_b $   extends  to a bounded operator  from  $BH^{p,s}_{w}$ to itself  that satisfies (2.4).
 \end{theorem}

\par In particularly, noticing Fact 1.7,  we have
\begin{corollary}\label{th 5}
Let $0<p \leq 1 $ and $   w\in A _{q,r}   $ with $ 0<q <p  $  and  $ 1<r <\infty  $, and $b\in BMO$. Then,

\par (i) $ \tilde{T}^{\Omega, *}_b $ extends  to a  bounded operator  from $H^{p}_{w}$ to $L^{p}_{w}$ that satisfies (2.3),

\par (ii)  if  $ w\in  P$, $ \tilde{T}^{\Omega, *}_b $ extends  to a  bounded operator from $H^{p}_{w}$ to itself that satisfies (2.5).
\end{corollary}

\par  \subsubsection{ } For the commutators that may not be known whether they are bounded on $L^p$ for $1<p<\infty$ but know the associated  operators are bounded  on $L^p$ for $1<p<\infty$, we can apply Theorem 2.5 and 2.7.

 \par  These commutators include the following.

\par $\bullet$  {\bf The commutator of the maximal Hilbert transform $H^*$}
 $$H_b^*f(x)= \sup_{\varepsilon>0} \int_{|x-y|>\varepsilon}\frac{(b(x)-b(y)) f(y)}{x-y}dy.$$

   \par $\bullet$  {\bf The commutator of the maximal Riesz transform $R^{j,*}$}
   $$R^{j,*}_b=\sup_{\varepsilon>0} \int_{|x-y|>\varepsilon} \frac{|x_j-y_j|}{|x-y|^{n+1}} (b(x)-b(y))  f(y)dy , j=1,2,\cdots,n.$$

  \par $\bullet$  {\bf The commutator of the maximal  rough singular integral operators $T^{\Omega,*}$}
   $$T^{\Omega, *}_b f(x)=\sup_{\varepsilon>0} \int_{|x-y|>\varepsilon} \frac{\Omega(\frac{x-y}{|x-y|})h(|x-y|)}{|x-y|^n} (b(x)-b(y))   f(y)dy $$
 with $h\in L^\infty ([0,\infty)), \Omega \in L^\infty (S^{n-1})$ and $\int _{{\bf S}^{n-1}}\Omega(u)d\sigma(u)=0$.

\par $\bullet$  {\bf The commutator of the maximal Fefferman type strongly singular multiplier operator $T^{F,*}$}
$$T^{F,*}_b f(x)= \sup_{\varepsilon>0} \int_{|x-y|>\varepsilon}\frac{e^{i|x-y|^{-\lambda}}}{|x-y|^n} \chi_E(|x-y|) (b(x)-b(y))  f(y)dy $$
with $0<\lambda<\infty$ and $\chi_E$ is the characteristic function of the unit interval $E=(0,1)\subset {\bf R} $.

\par $\bullet$   {\bf The commutator of the maximal Calder\'{o}n-Zygmund  operator $T^{CZ*}$}
$$T^{CZ,*}_bf(x)= \sup_{\varepsilon>0} \int_{|x-y|>\varepsilon}k(x,y) (b(x)-b(y))  f(y) dy,$$
 where $k(x,y)$ is a standard Calder\'{o}n-Zygmund  kernel.

\par $\bullet$  {\bf The commutators of  the maximal Carleson operator $C^*$}
$$
C^*_bf(x)=\sup_{\varepsilon>0}\sup_{\xi\in {\bf R}} \left|\int_{|x-y|>\varepsilon}\frac{e^{2\pi i \xi y}}{x-y} (b(x)-b(y))  f(y)dy\right|.
$$

\par $\bullet$  {\bf The commutators of the maximal Bochner-Riesz means    $B^{(n-1)/2,*}$ at the  critical index}
$$B^{(n-1)/2,*}_b f(x)=\sup _{0<R<\infty}|B^{(n-1)/2,R}((b(x)-b) f)(x)|.$$

\par $\bullet$ {\bf The commutator of the polynomial Carleson operator $ C^{d,n}$}
$$ C^{d,n}_b f(x)=\sup_{P\in \mathfrak{P}_{d,n}}  \left| \int_{{\bf R}^n} e^{  i P(x- y)} k(x-y) (b(x)-b(y)) f(y)dy\right|,$$
where $\mathfrak{P}_{d,n}$ is the class of all real-coefficient polynomials in
$n$ variables with no constant term and of degree less than or equal to $d, d\in {\bf N}$, and  $k$ is a suitable  Calder\'{o}n-Zygmund kernel on ${\bf R}^n$.

\par $\bullet$ {\bf The commutator of the maximal   oscillatory singular integrals $ T_{}^{P,*}$}
$$ T_{b}^{P,*}f(x)= \sup_{\varepsilon>0}  \left| \int_{|y|> \varepsilon} e^{  i P(x,y)} k(x-y) (b(x)-b(y)) f(y)dy\right|,$$
where $P:{\bf R}^n\times {\bf R}^n\rightarrow {\bf R}$    is a polynomial
of two variables, and  $k$ is a suitable  Calder\'{o}n-Zygmund kernel on ${\bf R}^n$.

\par
Denote  $T^{2}$  the  operators $H^*, R^{j*},T^{\Omega *}, T^{F*}, T^{CZ*},C^{*}, B^{(n-1)/2,*}, C^{d,1}$ and  $ T_{}^{P,*}$. And denote $T_b^{2}$  the commutator of $T^{2}$, i.e. $H^*_b, R^{j*}_b,T^{\Omega *}_b, T^{F*}_b, T^{CZ*}_b,C^{*}_b, B^{(n-1)/2,*}_b,  C^{d,1}_b$ and  $ T_{b}^{P,*}$.

\par  It is known that each $T^{2}$ is  bounded on $L^p$ with $1<p<\infty$, (see, \cite{St61970} for $H^*$ and $R^{j,*}$,  \cite{DR} for $T^{\Omega, *}$ , \cite{F2} for $T^{F,*}$ , \cite{Mey} for $T^{CZ,*}$, \cite{Carl,Hunt} for $C^{*}$, \cite{SS} for $B^{(n-1)/2,*}$, \cite{Lie} for $C^{d,1}$, and  \cite{KrL } for $ T_{}^{P,*}$).
 But, for every $T_b^{2}$,  it is open problem whether it is bounded on $L^p$ with $1<p<\infty$. While for $0<p\leq 1$, we knew that every $T_b^{2}$ fails to be  bounded on $H^p$, since  a bounded, compactly supported function $f$ belongs to $H^p$ if and only if it satisfies the vanishing  moment conditions $\int x^{\alpha}f(x)dx=0$ for all $|\alpha|\leq n(p^{-1}-1),$ (see \cite{S3},p129).

We have from Theorem 2.5(ii) and Theorem 2.7(ii) that
\begin{theorem}\label{th 6}
Let $0<p <\infty ,   w\in A _{q,r}   $ with $ 0<q <p  $  and  $ 1<r <\infty  $,  and $b\in BMO$.

\par (i) If      $  1   < s \leq \infty $ and $   rp/(r-1)  \leq s  $,
  then   each $T_b^{2}$   extends  to a  bounded operator from $BH^{p,s}_{w}$ to $L^{p}_{w}$ that satisfies (2.2).

\par (ii)  If  $ w\in  P$ and   $\max\{ p/q, rp/(r-1) \}< \tilde{s }<s< \infty $,
  then each $T_b^{2}$   extends  to a bounded from $BH^{p,s}_{w}$ to $BH^{p,\tilde{s}}_{w}$ that satisfies (2.6).
 \end{theorem}

\par  In particularly, noticing Fact 1.7,  we have
\begin{corollary}\label{th 5}
Let $0<p \leq 1 ,   w\in A _{q,r}  $ with $ 0<q <p  $  and  $ 1<r <\infty  $,  and $b\in BMO$. Then,

\par (i) each $T_b^{2}$ extends  to a  bounded operator  from $H^{p}_{w}$ to $L^{p}_{w}$ that satisfies (2.3).

\par (ii)  if  $ w\in  P$, then
each $T_b^{2}$ extends  to a  bounded operator from $H^{p}_{w}$ to itself that satisfies (2.5).
\end{corollary}

For the commutator of the polynomial Carleson operator $C^{d,n}_b$ of high dimensional,  Stein conjectured in \cite{Stein4} and  \cite{Stein5} that $ C^{d,n}$ is bounded on $L^p$ for any $1<p<\infty$, Lie proved the one dimensional case of this conjecture in  \cite{Lie}.
 We have

\begin{corollary}\label{th 8.1}
If  $C_{d,n}$ is bounded on $L^s$ with $1<s<\infty$, then, Theorem 2.15 and Corollary 2.16 hold for $C^{d,n}_b$.

\end{corollary}

  \par $\bullet$ {\bf The commutators of oscillatory singular integrals operators $T^{os}$}
$$T^{os}_bf(x)={\rm p.v.} \int_{{\bf
R}^n}e^{\lambda\Phi(x,y)}K(x,y)\varphi(x,y)(b(x)-b(y))f(y)dy,$$ where $K(x,y)$
is a Calder\'{o}n-Zygmund kernel,  $\varphi \in C_0^\infty({\bf
R}^n\times {\bf R}^n),$  $\lambda \in {\bf R}$, and $ \Phi(x,y)$ is
real-valued and   in the following cases:

\par ~~ (a)  $\Phi(x,y)=(Bx,y)$ is a  real bilinear
form, $ \varphi=1$ and $\lambda=1$,

\par ~~ (b)  $\Phi(x,y)=P(x,y) $ is a polynomial , $ \varphi=1$ and $\lambda=1$,

\par ~~ (c) $\Phi(x,y)$ is a real analytic function on suup($\varphi $).

\par
$T^{os}$ are bounded on $L^p$ with $1<p<\infty$,  (see, \cite{PhongStein} for case (a),  \cite{ RicciStein} for case (b) and
   \cite{Pan} for case (c)). We have from Theorem 2.5(i) and Theorem 2.7(i) that
\begin{theorem}\label{th 6}
Let $0<p <\infty ,  w\in A _{q,r}   $ with $ 0<q <p  $  and  $ 1<r <\infty  $,  and $b\in BMO$.

\par (i) If      $  1   < s \leq \infty $ and $   rp/(r-1)  \leq s  $,
  then   each $T^{os}_b$   extends  to a  bounded operator from $BH^{p,s}_{w}$ to $L^{p}_{w}$ that satisfies (2.2).

\par (ii)  If    $\max\{ p/q, rp/(r-1) \}< \tilde{s }<s< \infty $,
  then each $T^{os}_b$   extends  to a bounded from $BH^{p,s}_{w}$ to $BH^{p,\tilde{s}}_{w}$ that satisfies (2.6).
 \end{theorem}

\par  In particularly, noticing Fact 1.7,  we have
\begin{corollary}\label{th 5}
Let $0<p \leq 1 ,   w\in A _{q,r}  $ with $ 0<q <p  $  and  $ 1<r <\infty  $,  and $b\in BMO$. Then,
\par (i) each $T^{os}_b$ extends  to a  bounded operator  from $H^{p}_{w}$ to $L^{p}_{w}$ that satisfies (2.3),
\par (ii) each $T^{os}_b$ extends  to a  bounded operator from $H^{p}_{w}$ to itself that satisfies (2.5).
\end{corollary}

\section
 { Proofs of theorems  }
Let us first state some lemmas.

\begin{lemma}\label{lem_3}

Let $0<p <\infty ,  w\in  RH_r  $ with   $ 1<r <\infty  $,  and $  rp/(r-1) \leq s \leq \infty $. Then, for any cube $ Q$,
\begin{equation}\label{2.2}
\left(\int_{Q}w^{ s/(s-p)}\right)^{(s-p)/s} \leq C |Q|^{-p/s} w(Q ).
 \end{equation}
\end{lemma}

\par See \cite{L} for a proof of Lemma 3.1.

\begin{lemma}\label{lem_4}
Let $b\in BMO,  Q$ be a cube.
\par (i) Let $0<p <\infty ,  w\in  RH_r  $ with   $ 1<r <\infty  $,
  we have
\begin{equation}\label{2.2}
\int_{ 2^{i}Q  }|b(x)-b_Q|^pw(x)dx  \leq C\|b\|_{BMO}^p (i+1)^p w(2^{i}Q ), ~~ i=0,1,2,\cdots.
 \end{equation}
\par (ii) In particularly, let $1\leq  p<\infty$, we have
\begin{equation}\label{2.2}
\int_{ 2^{i}Q }|b(x)-b_Q|^pdx  \leq C\|b\|_{BMO}^p (i+1)^p |2^{i}Q|, ~~ i=0,1,2,\cdots.
 \end{equation}
\end{lemma}

\begin{proof}[Proof]
  Let us first prove (3.3). When $i=0$, the inequality
  $\int_{  Q }|b(x)-b_Q|^pdx  \leq C\|b\|_{BMO}^p   |  Q|$
  is well known, (see \cite{Gra}).
   Using this and  Minkowski inequality, we have,  for $i>0$,
\begin{eqnarray*}
&&\left ( \int_{ 2^{i}Q  }|b(x)-b_Q|^pdx\right) ^{1/p}
\\
&  \leq &
\left ( \int_{ 2^{i}Q  }|b(x)-b_{2^{i}Q}|^pdx\right) ^{1/p}
+ |b_{2^{i}Q}-b_{ Q}||b_{2^{i}Q}|^{1/p}
\\
&\leq &
|b_{2^{i}Q}|^{1/p}(\|b\|_{BMO}+|b_{2^{i}Q}-b_{ Q}|)
\\
&\leq&
C(i+1)  |b_{2^{i}Q}|^{1/p}\|b\|_{BMO},
 \end{eqnarray*}
since
$|b_{2^{i}Q}-b_{ Q}|\leq \sum_{j=0}^{i-1}|b_{2^{j+1}Q}-b_{ 2^{j}Q}|\leq Ci  \|b\|_{BMO}.$
 Thus, (3.3) holds.

 For (3.2),  taking  $\max\{ 1, rp/(r-1) \} \leq s < \infty $, by H\"{o}lder inequality, (3.3) and (3.1),  we have
\begin{eqnarray*}
&&
\int_{ 2^{i}Q  }|b(x)-b_Q|^pw(x)dx
\\
&  \leq &
\left ( \int_{ 2^{i}Q  }|b(x)-b_{Q}|^sdx\right) ^{p/s}
\left(\int_{2^{i}Q  }w^{ s/(s-p)}\right)^{(s-p)/s}
\\
&\leq&
 C\|b\|_{BMO}^p (i+1)^p w(2^{i}Q ).
 \end{eqnarray*}
Thus, (3.2) holds. This concludes the proof of Lemma 3.2.
\end{proof}

\begin{lemma} Let $0<p<\infty, 0<s\leq\infty$ and $w$ be a weight, we have,
\par
(i) if $|f(x)|\leq |g(x)|$, a.e., and $g\in  BH^{p,s}_{w}$, then $f\in  BH^{p,s}_{w}$ and $\| f\|_{BH^{p,s}_{w}} \leq  \| g\|_{BH^{p,s}_{w}}$;
\par
(ii) if  $f_i\in  BH^{p,s}_{w}, i=1,2,\cdots$, then  $\|\sum_{i=1}^{\infty} f_i\|^{\bar{p}}_{BH^{p,s}_{w}} \leq \sum_{i=1}^{\infty} \| f_i\|^{\bar{p}}_{BH^{p,s}_{w}}$;
\par
(iii) $ BH^{p,s}_{w}$ is complete.
\end{lemma}

\par See \cite{L} for a proof of Lemma 3.3.

\begin{lemma}\label{pro_2.9}
  Let $0<r< \infty$, we have ,
  \par (i) if  $w\in RH_{r}$, then, a.e. implies $w$-a.e.;
  \par (ii) if  $w\in RH_{r}\bigcap P$, then, $w$-a.e. implies a.e..
\end{lemma}
See \cite{L} for a proof of Lemma 3.4.

\par
The following known inequality will be used below, for $0<p<\infty$, weight $ w$, and  $b_i\in L^{p}_{w}$,$i=1,2,\cdots, $
\begin{equation}\label{2.1}
\|\sum_{i=1}^{\infty}  b_i \|^{\bar{p}}_{L^{p}_{w}} \leq  \sum_{i=1}^{\infty}\|  b_i \|^{\bar{p}}_{L^{p}_{w}}.
\end{equation}

\par Denote by $Q^{x_0}_{l} $ the cube centered at $x_0$ with side length $2l$. Clearly, $\lambda Q^{x_0}_{l} =Q^{x_0}_{\lambda l}$ for $\lambda>0$.
 Now, let us prove the theorems.

 \par In the following proof of the theorems, we always  suppose that  $T$ is defined for every $(p,s,w)$-block and is a  linear operator or sublinear  operator whose commutator satisfies (1.3), $b$ is a $BMO$ function, and  $h$ is a    $(p,s,w)$-block with  supp $h\subseteq Q=Q^{x_0}_{2^{m_0}}$.

 \par Denote $n^{1/2}=2^{l_0}$ and $ k_0=m_0+l_0 $, clearly $n^{1/2}Q=Q^{x_0}_{2^{k_0}}$.

\par   We need the following split of $T_bh$ to prove the theorems, for all $x\in {\bf R}^n$,
\begin{equation}\label{eq.7.5}
|T_bh(x)|\leq |T((b-b_{Q^{x_0}_{2^{k_0}}})h)(x)|+|b_{Q^{x_0}_{2^{k_0}}}-b(x)||Th(x)|.
\end{equation}
This can be seen from the following.
By H\"{o}lder inequality, (3.3) and the definition of $h$, we see that
\begin{eqnarray*}
 \|(b-b_{Q^{x_0}_{2^{k_0}}})h\|_{L^s}
 &\leq&  \|(b-b_{Q^{x_0}_{2^{k_0}}})\chi_{Q^{x_0}_{2^{k_0}}}\|_{L^{s'}} \| h\|_{L^s}
 \\
 &\leq& C|Q|^{1/s'}\|b\|_{BMO}|Q|^{1/s}w(Q)^{-1/p},
 \end{eqnarray*}
it follows $ \tilde{h}=(C|Q|^{1/s'}\|b\|_{BMO})^{-1}(b-b_{Q^{x_0}_{2^{k_0}}})h $ is a $(p,s,w)$-block, then,  given $\tilde{x}\in {\bf R}^n$, $ b(\tilde{x}) \in {\bf R}$, let
 \begin{equation*}
(b(\tilde{x})-b(y))h(y)
=
(b(\tilde{x})-b_{Q^{x_0}_{2^{k_0}}})h(y) +(b_{Q^{x_0}_{2^{k_0}}}-b(y))h(y)
=\mu_1 h(y) + \mu_2 \tilde{h}(y),
\end{equation*}
 where $\mu_1= b(\tilde{x})-b_{Q^{x_0}_{2^{k_0}}} \in {\bf R}, \mu_2= -C|Q|^{1/s'}\|b\|_{BMO}\in {\bf R}$, it is easy to check  $(\mu_1 h(y) + \mu_2 \tilde{h}(y))/(|\mu_1|+|\mu_2|)$
  is a $(p,s,w)$-block. Since $T$ is sublinear, then $T((b(\tilde{x})-b)h)(x)$ is well defined,
  and we have
  \begin{equation*}
|T(b(\tilde{x})-b)h)(x)|
\leq
|b(\tilde{x})-b_{Q^{x_0}_{2^{k_0}}}||Th(x)| +|T((b_{Q^{x_0}_{2^{k_0}}}-b)h)(x)|
\end{equation*}
for all $x\in {\bf R}^n$, taking $x=\tilde{x}$, it follows that   (3.5) holds.

\par Let us first prove Theorem 2.1 and Theorem 2.5.
\begin{proof}[Proof of Theorem 2.1  and Theorem 2.5 ]
To prove Theorem 2.1  and Theorem 2.5, by a standard argument, it is enough to show
 the following claim.

 {\bf Claim}: Under the conditions of Theorem 2.1 or Theorem 2.5, we have
  \begin{eqnarray}
 \|T_b h\|_{L^p_w} \leq C \|b\|_{BMO}
 \end{eqnarray}
  for each  $(p,s,w)$-block $h$.

\par  Let $\sum_{i=1}^{\infty} \lambda _i a_i  \in BH^{p,s}_{w}$, where
  each $a_i $ is a $(p,s,w)$-block  and $\sum_{i=1}^{\infty}|\lambda_i|^{\bar{p}} <\infty$.

  When $T$ is a linear operator, once (3.6) is proved, by (3.4),
 $\|T_b(\sum_{i=N_1}^{N_2} \lambda _i a_i) \|^{\bar{p}}_{L^{p}_{w}} \leq C \sum_{i=N_1}^{N_2} | \lambda _i|^{\bar{p}}$,
  we have that $\{\sum_{i=1}^{N} \lambda _i T_b a_i\}_{N=1}^{\infty}$ is a Cauchy sequence in $L^{p}_{w}$,
then, there is a unique $g\in L^{p}_{w}$, such that
 $   \sum_{i=1}^{N} \lambda _i T_ba_i \rightarrow g $ in $L^{p}_{w}$ as $N\rightarrow\infty$. Let $T_b(\sum_{i=1}^{\infty} \lambda _i a_i)=g$, and then (2.1) holds in   $L^{p}_{w}$. By using (3.6) again, we can prove for all $\delta>0$ that
$w(\{ x:|g-\sum_{i=1}^{\infty}\lambda _iT_ba_i|>\delta\})=0$ that implies that (2.1) holds $w$-a.e.. And then, using (3.4) and   (3.6),
  we get (2.2). Thus, we finish the proofs of the parts (i) of Theorem 2.1 and Theorem 2.5.

\par  When $T$ is a sublinear operator whose commutator  satisfies (1.3), once (3.6) is proved, using  (3.4), we get (2.2). Thus,  the parts (ii) of Theorem 2.1 and Theorem 2.5 hold.
\end{proof}

\par Next we need to  prove the claim.
To prove (3.6), we write, for the above $h$,
\begin{eqnarray}
 \|T_b h\|^p_{L^p_w} =\int_{Q^{x_0}_{2^{k_0+1}}}|T_b h|^pw+\int_{{\bf R}^n \backslash Q^{x_0}_{2^{k_0+1}}}|T_b h|^pw=: I+II.
 \end{eqnarray}

\begin{proof}[Proof of (3.6) under the conditions of Theorem 2.1]
 We need to estimate $I$ and $II$ in (3.7).
 \par For $I$, we notice $rp/(r-1)\leq s\leq \infty $. When $s<\infty$,  by H\"{o}lder inequality, the $L^s$ boundedness of $T_b$, (3.1), (1.1) and the  definition  of $h$, we have
\begin{eqnarray*}
I
&\leq&
\left(\int_{2n^{1/2}Q}|T_bh|^s\right)^{p/s}
\left(\int_{2n^{1/2}Q}w^{s/(s-p)}\right)^{(s-p)/s}
\\
&\leq&
C  \|b\|_{BMO}^p\left(\int_{Q}|h|^s\right)^{p/s}
 |Q|^{-p/s}w(Q)^{}
\\
&\leq &
C \|b\|_{BMO}^p.
 \end{eqnarray*}
When $s=\infty$,  using the $L^{\infty}$  boundedness of $T_b$ and the  definition  of $h$, we have $I\leq C \|b\|_{BMO}^p$.
\par
For $II$, we have by using (3.5) and (3.4) that
\begin{eqnarray*}
II^{\bar{p}/p}
&\leq &
\left(\int_{{\bf R}^n\backslash {Q^{x_0}_{2^{k_0+1}}}}|T((b-b_{Q^{x_0}_{2^{k_0}}})h)(x)|^pw(x)dx\right)^{\bar{p}/p}
\\
& & +\left(\int_{{\bf R}^n\backslash {Q^{x_0}_{2^{k_0+1}}}}(|b_{Q^{x_0}_{2^{k_0}}}-b(x)||Th(x)|)^pw(x)dx\right)^{\bar{p}/p}
\\
&=: &II_1^{\bar{p}/p}+II_2^{\bar{p}/p}.
 \end{eqnarray*}
 For $II_1$, we have,
\begin{eqnarray*}
II_1
 &\leq &
\|(b-b_{Q^{x_0}_{2^{k_0}}})h\|_{L^1}^p
\sum_{i=1}^{\infty}\int_{{Q^{x_0}_{2^{k_0+i+1}}}\backslash {Q^{x_0}_{2^{k_0+i}}}}|x-x_0|^{-np}w(x)dx
\\
& & ~~~~{\rm by} ~(1.2)
\\
&\leq &
C \|(b-b_{Q^{x_0}_{2^{k_0}}})h\|_{L^1}^p
\sum_{i=1}^{\infty} 2^{-npi}|Q^{x_0}_{2^{k_0}}|^{-p}w({Q^{x_0}_{2^{k_0+i+1}}})
\\
&\leq &
C \|(b-b_{Q^{x_0}_{2^{k_0}}})h\|_{L^1}^p|Q^{x_0}_{2^{k_0}}|^{-p}w(Q^{x_0}_{2^{k_0}})
\sum_{i=1}^{\infty} 2^{-n(p-q)i}
\\
&&~~~~~~~~~~~~~~~~~~~~~~~({\rm by}~ (1.1) )
\\
&\leq &
C \|b\|_{BMO}^p\|h\|_{L^s}^p|Q^{x_0}_{2^{k_0}}|^{p/s'}|Q^{x_0}_{2^{k_0}}|^{-p}w(Q^{x_0}_{2^{k_0}})
\\
&&~~~~~~~({\rm noticing } ~q<p, s>1, ~ {\rm by ~H\ddot{o}lder~inequality~ and~} (3.3))
\\
&\leq &
C \|b\|_{BMO}^p£¬
\\
&&~~~~~~~~~~~~~~~~~~~~~~~({\rm by~the ~ definition ~ of ~}h {\rm   ~ and ~} (1.1)).
 \end{eqnarray*}
 For $II_2$,  we have,
\begin{eqnarray*}
II_2
 &\leq &
C \|h\|_{L^1}^p
\sum_{i=1}^{\infty}
\int_{ Q^{x_0}_{2^{k_0+i+1}}\backslash Q^{x_0}_{2^{k_0+i}}}|b(x)-b_{Q^{x_0}_{2^{k_0}}}|^p|x-x_0|^{-np}w(x)dx
\\ &&~~~~~~~~~~~~~~~~~~~~~~~~~~~~~~~~~~~({\rm by}~(1.2))
\\
&\leq &
C \|h\|_{L^1}^p
\sum_{i=1}^{\infty}
2^{-npi}|Q^{x_0}_{2^{k_0}}|^{-p}
\int_{ Q^{x_0}_{2^{k_0+i+1}}}|b(x)-b_{Q^{x_0}_{2^{k_0}}}|^pw(x)dx
\\
&\leq &
C \|b\|_{BMO}^p
\|h\|_{L^1}^p
|Q^{x_0}_{2^{k_0}}|^{-p}
\sum_{i=1}^{\infty} 2^{-npi}i^pw({Q^{x_0}_{2^{k_0+i+1}}})
\\&&~~~~~~~~~~~~~~~~~~~~~~~~~~~~~~~~~~~({\rm by}~(3.2))
\\
&\leq &
C \|b\|_{BMO}^p
\|h\|_{L^1}^p|Q^{x_0}_{2^{k_0}}|^{-p}w(Q^{x_0}_{2^{k_0}})
\sum_{i=1}^{\infty} 2^{-n(p-q)i}i^p
\\ && ~~~~~~~~~~~~~~~~~~~~~~~({\rm by}~(1.1))
\\
&\leq &
C \|b\|_{BMO}^p\|h\|_{L^s}^p|Q^{x_0}_{2^{k_0}}|^{p/s'}|Q^{x_0}_{2^{k_0}}|^{-p}w(Q^{x_0}_{2^{k_0}})
\\
&&~~~~~~~({\rm noticing } ~q<p~,s>1,~ {\rm by ~H\ddot{o}lder~inequality})
\\
&\leq &
C\|b\|_{BMO}^p,
\\
&&~~~~~~~~~~~~~~~~~~~~~~~({\rm by~the ~ definition ~ of ~}h {\rm ~and}~(1.1)).
 \end{eqnarray*}
Then  $II\leq C \|b\|_{BMO}^p $. Thus,  (3.6) holds under the conditions of Theorem 2.1.
\end{proof}
 This concludes the proof of Theorem 2.1.
\begin{proof}[Proof of (3.6)  under the conditions of Theorem 2.5]
  We need to estimate $I$ and $II$ in (3.7) under the conditions of Theorem 2.5.
 We have $II \leq C\|b\|_{BMO}^p$   from the estimate of $II$ in the above proof of (3.6)  under the conditions of Theorem 2.1.
 Therefore, we need only to estimate $I$. We have by using (3.5) and (3.4) that
\begin{eqnarray*}
I^{\bar{p}/p}
&\leq &
\left(\int_{  {Q^{x_0}_{2^{k_0+1}}}}|T((b-b_{Q^{x_0}_{2^{k_0}}})h)(x)|^pw(x)dx\right)^{\bar{p}/p}
\\
& & +\left(\int_{  {Q^{x_0}_{2^{k_0+1}}}}(|b_{Q^{x_0}_{2^{k_0}}}-b(x)||Th(x)|)^pw(x)dx\right)^{\bar{p}/p} =:I_1^{\bar{p}/p}+I_2^{\bar{p}/p}.
 \end{eqnarray*}
\par  Now. let us estimate $I_1$. For $s<\infty$, we have that
\begin{eqnarray*}
I_1
 &=&
\int_{{Q^{x_0}_{2^{k_0+1}}}}|T((b-b_{Q^{x_0}_{2^{k_0}}})h)|^pw
\\
&\leq&
\left(\int_{{Q^{x_0}_{2^{k_0+1}}}}|T((b-b_{Q^{x_0}_{2^{k_0}}})h)|^{\tilde{s}}\right)^{p/\tilde{s}}
\left(\int_{{Q^{x_0}_{2^{k_0+1}}}}w^{\tilde{s}/(\tilde{s}-p)}\right)^{(\tilde{s}-p)/\tilde{s}}
\\
&&
~~~~~~~~~~~~~~~~({\rm by~H\ddot{o}lder~inequality~ for ~the~ index} ~\tilde{s}/p)
\\
&\leq&
C \left(\int_{{Q^{x_0}_{2^{k_0+1}}}}|(b-b_{Q^{x_0}_{2^{k_0}}})h|^{\tilde{s}}\right)^{p/\tilde{s}}
|Q^{x_0}_{2^{k_0+1}}|^{-p/\tilde{s}} w({Q^{x_0}_{2^{k_0+1}}})
\\
&& ~~~~~~~~({\rm by~the} ~L^{\tilde{s}}~{\rm boundedness ~of }~T   ~{\rm and~(3.1)~since}~w\in RH_r~{\rm and}~ r\geq \tilde{s}/(\tilde{s}-p))
\\
&\leq&
C |Q^{x_0}_{2^{m_0}}|^{-p/\tilde{s}} w({Q^{x_0}_{2^{m_0}}})
\\
&&
~~~~~~~~~~\times
\left(
\left(
\int_{{Q^{x_0}_{2^{k_0+1}}}}|h|^{\tilde{s}\cdot(s/\tilde{s})}
\right)^{1/(s/\tilde{s})}
\left(
\int_{{Q^{x_0}_{2^{k_0+1}}}}|b-b_{Q^{x_0}_{2^{k_0}}}|^{\tilde{s}\cdot(s/\tilde{s})'}
\right)^{1/(s/\tilde{s})'}
\right)^{p/\tilde{s}}
\\
&&~~~~~~({\rm by~H\ddot{o}lder~inequality~ for ~the~ index} ~s/\tilde{s}~{\rm and }~(1.1))
\\
&\leq &
C \|b\|_{BMO}^p\|h\|_{L^s}^{p}|Q^{x_0}_{2^{m_0}}|^{-p/s}w({Q^{x_0}_{2^{m_0}}})
\\
&&~~~~~~~~~~~~({\rm by}~(3.3) ~{\rm since ~  }1\leq \tilde{s}\cdot(s/\tilde{s})' <\infty ~{\rm which ~follows~ from~  }s/(s+1)\leq \tilde{s}<s)
\\
&\leq &
C \|b\|_{BMO}^p
\\
&&
~~~~~~~~~~~~~~~~~~~~~~~~~~~~~~~~~~~~~~~~~~~~~({\rm by}~\|h\|_{L^s} \leq |Q^{x_0}_{2^{m_0}}|^{1/s}w(Q^{x_0}_{2^{m_0}})^{-1/p} ).
 \end{eqnarray*}
For $s=\infty$, noticing $\left(\int_{{Q^{x_0}_{2^{k_0+1}}}}|(b-b_{Q^{x_0}_{2^{k_0}}})h|^{\tilde{s}}\right)^{p/\tilde{s}}\leq C\|h\|_{L^{\infty}}^p \|b\|_{BMO}^p|Q^{x_0}_{2^{m_0}}|^{p/\tilde{s}}$ by (3.3), then, according to the above estimate, we have
 $I_1\leq C\|h\|_{L^{\infty}}^p \|b\|_{BMO}^p w({Q^{x_0}_{2^{m_0}}})\leq C \|b\|_{BMO}^p$ by the definition of $h$.

 \par Next, let us estimate $I_2$.  For $s<\infty$,
 we have
\begin{eqnarray*}
I_2
 &=&
\int_{{2^{l_0+1}Q}}|(b-b_{Q^{x_0}_{2^{k_0}}})Th|^pw
\\
&\leq&
\left(\int_{{Q^{x_0}_{2^{k_0+1}}}}|(b-b_{Q^{x_0}_{2^{k_0}}})Th|^{\tilde{s}}\right)^{p/\tilde{s}}
\left(\int_{{Q^{x_0}_{2^{k_0+1}}}}w^{\tilde{s}/(\tilde{s}-p)}\right)^{(\tilde{s}-p)/\tilde{s}}
\\
&&
~~~~~~~~~~~~~~~~({\rm by~H\ddot{o}lder~inequality~ for ~the~ index} ~\tilde{s}/p)
\\
&\leq&
C|Q^{x_0}_{2^{m_0}}|^{-p/\tilde{s}} w({Q^{x_0}_{2^{m_0}}})
\\
&&
~~~~~~~~~~\times
 \left(
\left(
\int_{{Q^{x_0}_{2^{k_0+1}}}}|Th|^{\tilde{s}\cdot(s/\tilde{s})}
\right)^{1/(s/\tilde{s})}
\left(
\int_{{Q^{x_0}_{2^{k_0+1}}}}|b-b_{Q^{x_0}_{2^{k_0}}}|^{\tilde{s}\cdot(s/\tilde{s})'}
\right)^{1/(s/\tilde{s})'}
\right)^{p/\tilde{s}}
\\
&&~~~~~~({\rm by~H\ddot{o}lder~inequality~ for ~the~ index} ~s/\tilde{s},~ {\rm(3.1)  } ~ {\rm and }~(1.1))
\\
&\leq &
C \|b\|_{BMO}^p\|h\|_{L^s}^{p}|Q^{x_0}_{2^{m_0}}|^{-p/s} w({Q^{x_0}_{2^{m_0}}})
\\
&& ~~~~~~~~({\rm by~the ~boundedness ~of }~T {\rm ~on} ~L^{s} ~{\rm and~ (3.3)~  }{\rm since ~  }1\leq \tilde{s}\cdot(s/\tilde{s})' <\infty)
\\
&\leq &
C \|b\|_{BMO}^p
\\
&&
~~~~~~~~~~~~~~~~~~~~~~~~~~~~~~~~~~~~~~~~~~~~~({\rm by~the ~definition~of~} h ).
 \end{eqnarray*}

 \par  For $s=\infty$, noticing $\|Th\|_{L^{\infty}}\leq C\|h\|_{L^{\infty}}$, we have by(3.3) that
  $$\left(\int_{{Q^{x_0}_{2^{k_0+1}}}}|(b-b_{Q^{x_0}_{2^{k_0}}})Th|^{\tilde{s}}\right)^{p/\tilde{s}}\leq C\|h\|_{L^{\infty}}^p \|b\|_{BMO}^p|Q^{x_0}_{2^{m_0}}|^{p/\tilde{s}},$$ then, according to the above estimate, we have
 $I_2\leq C\|h\|_{L^{\infty}}^p \|b\|_{BMO}^p w({Q^{x_0}_{2^{m_0}}})\leq C \|b\|_{BMO}^p$ by the definition of $h$.

\par Thus, we have proved (3.6) under the conditions of Theorem 2.5.
\end{proof}
 This concludes the proof of Theorem 2.5.

\par To prove  Theoorem 2.3 and Theorem 2.7, we need the following molecular characterization of $BH^{p,s}_{w}$.

\begin{definition}\label{def 5.1}
  Let $0<p <\infty,    w\in A _{q,r}   $ with $ 0<q <p  $  and  $ 1<r <\infty  $,
   and $\max\{ rp/(r-1), p/q\}< s\leq \infty$.
  Set $0<\varepsilon <1 -q/p,
c =1-q/p- \varepsilon,$ and $ d=1-1/s- \varepsilon.$
  A function $M(x)\in L^s $
 is said to be a  $ (p, s,q,w, \varepsilon )$-molecule (centered at  $x_0$), if

\par (i)~~~~$M(x)\left( |Q^{x_0}_{|x-x_0|}|^{-1/s}w(Q^{x_0}_{|x-x_0|})^{1/p}\right)^{d/(d-c)}\in L^s ,$

\par (ii)~~~~$\|M\|^{c/d}_{L^{s}   }
\left \|
M(x)\left( |Q^{x_0}_{|x-x_0|}|^{-1/s}w(Q^{x_0}_{|x-x_0|})^{1/p}\right)^{d/(d-c)}
\right \|^{1-c/d}_{L^{s}   }
\equiv \Re (M)<\infty.$
\end{definition}

 \par  For convenience, we will denote $\Re (M)$ by $\Re _s(M)$.

\begin{proposition}\label{th_5.3}
  Let $0<p <\infty,    w\in A _{q,r}   $ with $ 0<q <p  $  and  $ 1<r <\infty  $,
 and $\max\{ rp/(r-1), p/q\}< s\leq \infty$.
Then,
   every  $ (p, s, q, w, \varepsilon )$-molecular $M(x)$
   is in $BH^{p,s}_{w}$, with $\|M \|^{
}_{BH^{p,s}_{w}}\leq C \Re(M)$, i.e.
\begin{equation*}\label{5.6}
M(x)=\sum \limits_{k=0}^{\infty}\lambda_ka_k(x), {\rm ~converges ~for~ all~} x\in {\bf R}^n,
  \end{equation*}
where each $a_k$ is a $ (p, s, w )$-block and $\left(\sum
\limits_{k=0}^{\infty}|\lambda_k|^{\bar{p}}\right)^{1/\bar{p}}\leq C \Re(M)$,  the constant $C$ is
independent of $M$.
\end{proposition}

\par
See \cite{L} for the above definition and proposition.

\par Now let us  prove Theorem 2.3 and Theorem 2.7.
\begin{proof}[Proof of Theorem 2.3  and Theorem 2.7 ]
By a standard argument, it is enough to show
 the following claim.
\par
 {\bf Claim}: Let $0<p <\infty ,    w\in A _{q,r}   $ with $ 0<q <p  $  and  $ 1<r <\infty  $,   $h$ be a $(p,s,w)$-block, and $T$     satisfy (1.2).

\par
  (i) If $\max\{ p/q, rp/(r-1) \} <  s  \leq \infty $ and
             $ \|T_b h\|_{L^{ s}}\leq C\|b\|_{BMO} \| h\|_{L^{s}} $, (i.e. under the conditions of Theorem 2.3),
       then
\begin{equation}
 \|T_b h\|_{BH^{p,s}_{w}} \leq C \|b\|_{BMO}
 \end{equation}
\par
  (ii) If $\max\{ p/q, rp/(r-1) \} < \tilde{s} < s  \leq \infty $,
             $ \|T h\|_{L^{ s}}\leq C \| h\|_{L^{s}} $
          and
          $ \|T h\|_{L^{\tilde{s}}}\leq C \| h\|_{L^{\tilde{s}}}, $ (i.e. under the conditions of Theorem 2.7),
 then
 \begin{equation}
 \|T_b h\|_{BH^{p,\tilde{s}}_{w}} \leq C \|b\|_{BMO}.
 \end{equation}

\par  Let $\sum_{i=1}^{\infty} \lambda _i a_i  \in BH^{p,s}_{w}$, where
  each $a_i $ is a $(p,s,w)$-block and   $\sum_{i=1}^{\infty}|\lambda_i|^{\bar{p}} <\infty$.

 \par When $T$ is a linear operator,     once (3.8) (or (3.9)) are proved, by using Lemma 3.3(ii), we see that
    $\{\sum_{i=1}^{N} \lambda _i T_b a_i\}_{N=1}^{\infty}$ is a Cauchy sequence in $BH^{p,s}_{w}$ (or $BH^{p,\tilde{s}}_{w}$),
then, by Lemma 3.3(iii), there is a unique $g\in BH^{p,s}_{w}$ (or $BH^{p,\tilde{s}}_{w}$),
such that
 $   \sum_{i=1}^{N} \lambda _i T_ba_i \rightarrow g $ in $BH^{p,s}_{w}$ (or $BH^{p,\tilde{s}}_{w}$) as $N\rightarrow\infty$,
 Let $T_b(\sum_{i=1}^{\infty} \lambda _i a_i)=g$,  then (2.1) holds in  $BH^{p,s}_{w}$ (or $BH^{p,\tilde{s}}_{w}$),
 also  in $L^{p}_{w}$ since $BH^{p,s}_{w}$ (and $BH^{p,\tilde{s}}_{w}$) in $L^{p}_{w}$ (see \cite {L}),
 by Theorem 2.1(i), (2.1) holds $w$-a.e..
  Using Lemma 3.3(ii)
and (3.8) (or (3.9)),  we get (2.4) (or (2.6)).  Thus, we finish the proofs of the parts (i) of Theorem 2.3 (or Theorem 2.7).

\par  When $T$ is a sublinear operator satisfying (1.3),  once (3.8) is proved,
  by Lemma 3.3(ii), we have
\begin{equation}\label{7.2}
\| \sum|\lambda_j||T_ba_j(x)|\|^{\bar{p}}_{BH^{p,s}_{w}}\leq  \sum|\lambda_j|^{\bar{p}}\||T_ba_j(x)|\|^{\bar{p}}_{BH^{p,s}_{w}}
\leq  C\|f\|^{\bar{p}}_{BH^{p,s}_{w}}.
\end{equation}
Noticing  $w\in  P$, we see from Lemma 3.4(ii) that (1.3) hold a.e., and then, using Lemma 3.3(i) and (3.10)
 we get (2.4).
  Similarly, we get (2..6). Thus,  the parts (ii) of Theorem 2.3 and Theorem 2.7 hold.
\end{proof}
\par Next we need to  prove (3.8) and (3.9).
To prove these, by the molecular theorem  (Proposition 3.6),  it  suffices to check for every $(  p,s,w)$-block $h$ that $Th$ is
 a  $ (p, s, q, w,\varepsilon )$-molecular and $\Re_s (Th)\leq C \|b\|_{BMO}$   under the condition of Theorem 2.3,
  and that $Th$ is a  $ (p, \tilde{s}, q, w,\varepsilon )$-molecular and $\Re_{\tilde{s}} (Th)\leq C \|b\|_{BMO}$  under the condition of Theorem 2.7.

\par
To estimate $\Re_s (Th)$,  we need estimate
$$  A_s=:\left\|T_bh(x)\left( |Q^{x_0}_{|x-x_0|}|^{-1/s}w(Q^{x_0}_{|x-x_0|})^{1/p}\right)^{d/(d-c)}\right\|^{}_{L^{s}},$$
  and
 $$ B_s=:\|T_bh\|^{}_{L^{s}},$$
for every $(  p,s,w)$-block $h$.

 To estimate  $\Re_{\tilde{s}} (Th)$, we need estimate $A_{\tilde{s}}$ and $B_{\tilde{s}}$   for every $(  p,s,w)$-block $h$.

\par
To estimate $A_s$ and $A_{\tilde{s}}$, we suppose that
 supp $h\subset Q=Q^{x_0}_{2^{m_0}} $.
 We write,
\begin{equation}
   A_s^s
  =
 \left(\int _{Q^{x_0}_{2^{k_0+1}}}+\int _{(Q^{x_0}_{2^{k_0+1}})^c} \right)
|T_bh(x)|^s \left( |Q^{x_0}_{|x-x_0|}|^{-1/s}w(Q^{x_0}_{|x-x_0|})^{1/p}\right)^{ds/(d-c)} dx
\end{equation}
 \par $~$  $~~~~~~~~~~~~~=~:~  J_s+JJ_s,$
 \\similarly, write
 $$A_{\tilde{s}}^{\tilde{s}} = :J_{\tilde{s}} +JJ_{\tilde{s}}.$$

\begin{proof}[Proof of (3.8)   under the conditions of Theorem 2.3]
Let $b$ be a $BMO$ function,  $h$ be a $(p,s,w)$-block with supp$h\subset Q=Q^{x_0}_{2^{m_0}} $.
 We need to estimate $A_s$   and $B_s$.

 For $B_s$, by the boundedness of $T_b$ on $L^s$, we have
$$
B_s=\|T_bh \|^{ }_{L^{s}   } \leq  C \|b\|_{BMO} \| h  \|^{ }_{L^{s}   } .
$$

\par
 For $A_s$, we need to estimate $J_s$ and $JJ_s$, $s\leq \infty$.

 \par   For $J_s$, by using for $x\in Q^{x_0}_{2^{k_0+1}}$ that
   $$
   |Q^{x_0}_{|x-x_0|}|^{-1/s}w(Q^{x_0}_{|x-x_0|})^{1/p}\leq |Q^{x_0}_{2^{k_0}}|^{-1/s}w(Q^{x_0}_{2^{k_0}})^{1/p}
   $$
   which follows from (6.4)  in \cite{L}, and the  $L^s$-boundedness of $T_b$, we have
   $$J_s\leq C\|b\|_{BMO}^s\left(
|Q^{x_0}_{2^{k_0}}|^{-1/s}w(Q^{x_0}_{2^{k_0}})^{1/p}
\right)^{ds/(d-c)}
 \|h \|^{s}_{L^{s}   }.$$

\par For $JJ_s$, noticing
that
\begin{equation}\label{7.9}
w( Q^{x_0}_{|x-x_0|}) \leq C(|Q^{x_0}_{2^{k_0}}|^{-1}|Q^{x_0}_{|x-x_0|}|)^{q}w(Q^{x_0}_{2^{k_0-1}})
\end{equation}
for $x\in (Q^{x_0}_{2^{k_0+1}})^c$, which follows from (1.1), and that $d>0, d-c>0$,
 we have
\begin{eqnarray*}
&&JJ_s
\\
&=&
\int _{(Q^{x_0}_{2^{k_0+1}})^c}
|T_b h(x)|^s \left( |Q^{x_0}_{|x-x_0|}|^{-1/s}w(Q^{x_0}_{|x-x_0|})^{1/p}\right)^{ds/(d-c)} dx
\\
& \leq &
C
\int _{(Q^{x_0}_{2^{k_0+1}})^c}
|T_b h(x)|^s \left(
|Q^{x_0}_{2^{k_0}}|^{-q/p} |Q^{x_0}_{|x-x_0|}|^{-1/s+q/p}
w(Q^{x_0}_{2^{k_0-1}})^{1/p}\right)^{ds/(d-c)} dx
\\
& &
~~~~~~~~~~~~~~~~~~~~~~~~({\rm by}~(3.12))
\\
& \leq &
C
\left(
|Q^{x_0}_{2^{k_0}}|^{-q}
w(Q^{x_0}_{2^{k_0}})\right)^{ds/[p(d-c)]}
\int _{(Q^{x_0}_{2^{k_0+1}})^c}
|T_b h(x)|^s  |x-x_0|^{dns}
 dx
\\
& &
~~~~~~~~~~~~~~~~~~~~~~~~({\rm noticing} ~d-c=q/p-1/s {\rm ~ and ~} |Q^{x_0}_{|x-x_0|}|=C|x-x_0|^n)
\\
& \leq &
C( JJ_{s1}^{1/s}+JJ_{s2}^{1/s})^s
 \\
& &
~~~~~~~~~~~~~~~~~~~~~~~~({\rm by} ~(3.5) {\rm ~and~Minkowski~ inequality} ),
\end{eqnarray*}
where
\begin{eqnarray*}
JJ_{s1}=\left(
|Q^{x_0}_{2^{k_0}}|^{-q}
w(Q^{x_0}_{2^{k_0}})\right)^{\frac{ds}{p(d-c)}}
\int _{(Q^{x_0}_{2^{k_0+1}})^c}
|T((b-b_{Q^{x_0}_{2^{k_0}}}) h)(x)|^s  |x-x_0|^{dns}
 dx,
\end{eqnarray*}
and
\begin{eqnarray*}
JJ_{s2}=\left(
|Q^{x_0}_{2^{k_0}}|^{-q}
w(Q^{x_0}_{2^{k_0}})\right)^{\frac{ds}{p(d-c)}}
\int _{(Q^{x_0}_{2^{k_0+1}})^c}
|(b(x)-b_{Q^{x_0}_{2^{k_0}}}) T h(x)|^s  |x-x_0|^{dns}
 dx.
\end{eqnarray*}
For $JJ_{s1}$, we have
\begin{eqnarray*}
JJ_{s1}
& \leq &
C
\left(
|Q^{x_0}_{2^{k_0}}|^{-q}
w(Q^{x_0}_{2^{k_0}})\right)^{ds/[p(d-c)]}
\\ & &
\times
\|(b-b_{Q^{x_0}_{2^{k_0}}})h\|_{L^1}^s
\int _{(Q^{x_0}_{2^{k_0+1}})^c}
  |x-x_0|^{(d-1)ns}
 dx
\\& &~~~~~~~~~~~~~~~~~~~
~~~~~({\rm by}~(1.2))
\\
& \leq &
C \|b\|_{BMO}^s
\left(
|Q^{x_0}_{2^{k_0}}|^{-q}
w(Q^{x_0}_{2^{k_0}})\right)^{ds/[p(d-c)]}
\\
&&
\times
\|h\|_{L^s}^s
|Q^{x_0}_{2^{k_0}}|^{(1-1/s)s}
|Q^{x_0}_{2^{k_0}}|^{s(d-1+1/s)}
 \\
& &
~~~~~~~~~~~~~~~({\rm by ~H\ddot{o}lder's~inequality,~} (3.3)~ {\rm and ~
noticing }~d-1+1/s=-\varepsilon <0)
\\
& = & C\|b\|_{BMO}^s \left(
|Q^{x_0}_{2^{k_0}}|^{-1/s}w(Q^{x_0}_{2^{k_0}})^{1/p}
\right)^{ds/(d-c)}
 \|h \|^{s}_{L^{s}   }
 \\& &~~~~~~~~~~~~~
 ~~~~~~~~~~~({\rm noticing} ~d-c=q/p-1/s).
\end{eqnarray*}
For $JJ_{s2}$, we have
\begin{eqnarray*}
JJ_{s2}
& \leq &
C
\left(
|Q^{x_0}_{2^{k_0}}|^{-q}
w(Q^{x_0}_{2^{k_0}})\right)^{ds/[p(d-c)]}
\\ & &
\times
\|h\|_{L^1}^s
\sum_{i=1}^{\infty}
\int_{ Q^{x_0}_{2^{k_0+i+1}}\backslash Q^{x_0}_{2^{k_0+i}}}
 |b(x)-b_{Q^{x_0}_{2^{k_0}}}|^s |x-x_0|^{(d-1)ns}
 dx
\\& &~~~~~~~~~~~~~~~~~~~
~~~~~({\rm by}~(1.2))
\\
& \leq &
C
\left(
|Q^{x_0}_{2^{k_0}}|^{-q}
w(Q^{x_0}_{2^{k_0}})\right)^{ds/[p(d-c)]}
\|h\|_{L^s}^s |Q^{x_0}_{2^{k_0}}|^{(1-1/s)s}
\\
&&
\times
\sum_{i=1}^{\infty}
|Q^{x_0}_{2^{k_0+i}}|^{s(d-1)}
\int_{ Q^{x_0}_{2^{k_0+i+1}}}
 |b(x)-b_{Q^{x_0}_{2^{k_0}}}|^s dx
 \\
& &
~~~~~~~~~~~~~~~({\rm by ~H\ddot{o}lder's~inequality})
\\
& \leq &
C
\left(
|Q^{x_0}_{2^{k_0}}|^{-q}
w(Q^{x_0}_{2^{k_0}})\right)^{ds/[p(d-c)]}
\|h\|_{L^s}^s |Q^{x_0}_{2^{k_0}}|^{(1-1/s)s}
\\
&&
\times
\|b\|_{BMO}^s
\sum_{i=1}^{\infty}
|Q^{x_0}_{2^{k_0+i}}|^{s(d-1)+1}i^s
 \\
& &
~~~~~~~~~~~~~~~({\rm by  ~} (3.3))
\\
& \leq &
C\|b\|_{BMO}^s
\left(
|Q^{x_0}_{2^{k_0}}|^{-q}
w(Q^{x_0}_{2^{k_0}})\right)^{ds/[p(d-c)]}
\|h\|_{L^s}^s |Q^{x_0}_{2^{k_0}}|^{(1-1/s)s}
|Q^{x_0}_{2^{k_0}}|^{s(d-1)+1}
 \\
& &
~~~~~~~~~~~~~~~( {\rm
noticing }~d-1+1/s=-\varepsilon <0)
\\
& = & C\|b\|_{BMO}^s \left(
|Q^{x_0}_{2^{k_0}}|^{-1/s}w(Q^{x_0}_{2^{k_0}})^{1/p}
\right)^{ds/(d-c)}
 \|h \|^{s}_{L^{s}   }
 \\& &~~~~~~~~~~~~~
 ~~~~~~~~~~~({\rm noticing} ~d-c=q/p-1/s).
\end{eqnarray*}
Thus,
\begin{equation}
JJ_{s}\leq C\|b\|_{BMO}^s \left(
|Q^{x_0}_{2^{k_0}}|^{-1/s}w(Q^{x_0}_{2^{k_0}})^{1/p}
\right)^{ds/(d-c)}
 \|h \|^{s}_{L^{s}   }.
 \end{equation}
\par Combining this with the  estimate of $J_{s}$, and using (1.1), we have
\begin{eqnarray*}
A_{s} &\leq &
C\|b\|_{BMO}
\left(
|Q^{x_0}_{2^{m_0}}|^{-1/s}w(Q^{x_0}_{2^{m_0}})^{1/p}
\right)^{b/(d-c)}
 \|h \|^{}_{L^{s}   }.
\end{eqnarray*}

\par Then, Combining the above estimate of $A_{s}$ and $B_{s}$,
 noticing $c/d>0$ and $1-c/d> 0,$ we have
\begin{eqnarray*}
\Re_{s} (T_bh)\leq C\|b\|_{BMO}\| h  \|^{ }_{L^{s}   }|Q^{x_0}_{2^{m_0}}|^{-1/s}w(Q^{x_0}_{2^{m_0}})^{1/p}\leq C\|b\|_{BMO},
\end{eqnarray*}
by the definition of $h$.
Thus, (3.8) is proved under the conditions of Theorem 2.3.
\end{proof}
This concludes the proof of Theorem 2.3.

\begin{proof}[Proof of (3.9)  under the conditions of Theorem 2.7]

 Let $b$ be a $BMO$ function,  $h$ be a $(p,s,w)$-block with
 supp $h\subseteq Q $. We need to estimate $A_{\tilde{s}}$  and $B_{\tilde{s}}$.  Once
  \begin{equation}
  B_{\tilde{s}} \leq C  \|b\|_{BMO}|Q|^{1/\tilde{s}}w(Q)^{-1/p}.
   \end{equation}
and
 \begin{equation}
 A_{\tilde{s}}\leq
 C \|b\|_{BMO}
  \left(
|Q|^{-1/\tilde{s}}w(Q)^{1/p}
\right)^{b/(d-c)}
|Q|^{1/\tilde{s}}w(Q)^{-1/p}
 \end{equation}
are established, it is easy to get $\Re_{\tilde{s}} (T_bh)\leq C\|b\|_{BMO}$.

\par  Now, let us prove (3.14) and (3.15).
 Let us first prove  (3.14). Noticing $1<\tilde{s}<\infty$,
 we write
\begin{eqnarray*}
 \|T_bh\|^{\tilde{s}}_{L^{\tilde{s}}}
 =
 \left(\int _{Q^{x_0}_{2^{k_0+1}}}+\int _{(Q^{x_0}_{2^{k_0+1}})^c} \right)
|T_bh(x)|^{\tilde{s}}  dx
  = : \tilde{J}+\widetilde{JJ}.
\end{eqnarray*}

\par For $\tilde{J}$, we have by (3.5) and Minkowski inequality that
\begin{eqnarray*}
\tilde{J}^{1/\tilde{s}}
&\leq &
\left(\int_{Q^{x_0}_{2^{k_0+1}}}|T((b-b_{Q^{x_0}_{2^{k_0}}})h)(x)|^{\tilde{s}}dx\right)^{1/\tilde{s}}
 \\
& &+\left(\int_{Q^{x_0}_{2^{k_0+1}}}(|b_{Q^{x_0}_{2^{k_0}}}-b(x)||Th(x)|)^{\tilde{s}}dx\right)^{1/\tilde{s}}
 =: \tilde{J}_1^{1/\tilde{s}}+\tilde{J}_2^{1/\tilde{s}}.
 \end{eqnarray*}
For $\tilde{J}_1$, we have by the boundedness  of  $ T$ on $L^{\tilde{s}}$ that
\begin{eqnarray*}
\tilde{J}_1
 =
\int_{Q^{x_0}_{2^{k_0+1}}}|T((b-b_{Q^{x_0}_{2^{k_0}}})h)|^{\tilde{s}}
\leq
\int_{Q^{x_0}_{2^{k_0+1}}}|(b-b_{Q^{x_0}_{2^{k_0}}})h|^{ {\tilde{s}}}.
\end{eqnarray*}
It follows that, for $1<s<\infty$,
\begin{eqnarray*}
\tilde{J}_1
&\leq&
C
\left(
\int_{Q^{x_0}_{2^{k_0+1}}}|h|^{{\tilde{s}}\cdot(s/\tilde{s})}
\right)^{1/(s/\tilde{s})}
\left(
\int_{Q^{x_0}_{2^{k_0+1}}}|b-b_{Q^{x_0}_{2^{k_0}}}|^{{\tilde{s}}\cdot(s/\tilde{s})'}
\right)^{1/(s/\tilde{s})'}
\\
&&
~~~~~~~~~~~~~~~~({\rm by~H\ddot{o}lder~inequality~ for ~the~ index} ~s/\tilde{s})
\\
&\leq&
C \|b\|_{BMO}^{\tilde{s}}\|h\|_{L^{s}}^{\tilde{s}}|Q^{x_0}_{2^{k_0}}|^{1/(s/\tilde{s})'}
\\
&&~~~~~~~~~~~~({\rm by}~(3.3) ~{\rm since ~  }1\leq \tilde{s}\cdot(s/\tilde{s})' <\infty ~{\rm which ~follows~ from~  }s/(s+1)\leq \tilde{s}<s)
\\
&\leq &
C  \|b\|_{BMO}^{\tilde{s}}|Q^{x_0}_{2^{m_0}}|w(Q^{x_0}_{2^{m_0}})^{-\tilde{s}/p}
\\
&&
~~~~~~~~~~~~~~~~({\rm by~the ~definition~ of ~} h ),
 \end{eqnarray*}
for $s=\infty$, by (3.3) and the definition of $h$,
\begin{eqnarray*}
\tilde{J}_1
&\leq&
\int_{Q^{x_0}_{2^{k_0+1}}}|(b-b_{Q^{x_0}_{2^{k_0}}})h|^{ {\tilde{s}}}
\leq
\|h\|_{L^{\infty}}^{\tilde{s}}
\int_{Q^{x_0}_{2^{k_0+1}}}|(b-b_{Q^{x_0}_{2^{k_0}}})|^{ {\tilde{s}}}
\\
&\leq &
C  \|b\|_{BMO}^{\tilde{s}}|Q^{x_0}_{2^{m_0}}|w(Q^{x_0}_{2^{m_0}})^{-\tilde{s}/p}
.
\end{eqnarray*}
For $\tilde{J}_2$, we have that, for $s<\infty$,
\begin{eqnarray*}
\tilde{J}_2
&\leq&
C
\left(
\int_{Q^{x_0}_{2^{k_0+1}}}|Th|^{{\tilde{s}}\cdot(s/\tilde{s})}
\right)^{1/(s/\tilde{s})}
\left(
\int_{Q^{x_0}_{2^{k_0+1}}}|b-b_{Q^{x_0}_{2^{k_0}}}|^{{\tilde{s}}\cdot(s/\tilde{s})'}
\right)^{1/(s/\tilde{s})'}
\\
&&
~~~~~~~~~~~~~~~~({\rm by~H\ddot{o}lder~inequality~ for ~the~ index} ~s/\tilde{s})
\\
&\leq&
C
\left(
\int_{Q^{x_0}_{2^{k_0+1}}}|h|^{{\tilde{s}}\cdot(s/\tilde{s})}
\right)^{1/(s/\tilde{s})}
\left(
\int_{Q^{x_0}_{2^{k_0+1}}}|b-b_{Q^{x_0}_{2^{k_0}}}|^{{\tilde{s}}\cdot(s/\tilde{s})'}
\right)^{1/(s/\tilde{s})'}
\\
&&
~~~~~~~~~~~~~~~~({\rm by~the~boundedness ~of ~} T~ {\rm on }~L^{s})
\\
&\leq &
C  \|b\|_{BMO}^{\tilde{s}}|Q^{x_0}_{2^{m_0}}|w(Q^{x_0}_{2^{m_0}})^{-\tilde{s}/p}
\\
&&~~~~~~~~~~~~({\rm by}~(3.3) ~{\rm since ~  }1\leq \tilde{s}\cdot(s/\tilde{s})' <\infty, ~{\rm and ~the ~definition~ of ~} h  )
.
 \end{eqnarray*}
 for $s=\infty$,  by the boundedness of $T$  on $L^{\infty}$, (3.3) and the definition of $h$,
\begin{eqnarray*}
\tilde{J}_2
 =
\int_{Q^{x_0}_{2^{k_0+1}}}|(b-b_{Q^{x_0}_{2^{k_0}}})Th|^{\tilde{s}}
\leq
C  \|b\|_{BMO}^{\tilde{s}}|Q^{x_0}_{2^{m_0}}|w(Q^{x_0}_{2^{m_0}})^{-\tilde{s}/p}.
\end{eqnarray*}
\par For $\widetilde{JJ}$, we have by (3.5) and Minkowski inequality that
\begin{eqnarray*}
\widetilde{JJ}^{1/\tilde{s}}
&\leq &
\left(\int_{(Q^{x_0}_{2^{k_0+1}})^c}|T((b -b_{Q^{x_0}_{2^{k_0}}})h )(x)|^{\tilde{s}}dx\right)^{1/\tilde{s}}
\\
& & +\left(\int_{(Q^{x_0}_{2^{k_0+1}})^c}(|b_{Q^{x_0}_{2^{k_0}}}-b(x)||Th(x)|)^{\tilde{s}}dx \right)^{1/\tilde{s}}=:\widetilde{JJ}_1^{1/\tilde{s}}+\widetilde{JJ}_2^{1/\tilde{s}}.
 \end{eqnarray*}
For $\widetilde{JJ}_1$, we have
\begin{eqnarray*}
\widetilde{JJ}_1
 &\leq &
\|(b-b_{Q^{x_0}_{2^{k_0}}})h\|_{L^1}^{\tilde{s}}
\int_{(Q^{x_0}_{2^{k_0+1}})^c}|x-x_0|^{-n\tilde{s}}dx
\\
&&~~~~~~~~~~~~~~~~~~~~~~~~~~~~~~~~~~~({\rm by}~(1.2))
\\
&\leq &
C \|(b-b_{Q^{x_0}_{2^{k_0}}})h\|_{L^1}^{\tilde{s}}
|Q^{x_0}_{2^{k_0}}|^{1-\tilde{s}}
\\
&&~~~~~~~( {\rm noticing}~ 1<\tilde{s}<\infty)
\\
&\leq &
C\|b\|_{BMO}^{\tilde{s}} |Q^{x_0}_{2^{m_0}}|w(Q^{x_0}_{2^{m_0}})^{-\tilde{s}/p}
\\
&&
~~~~~~~~~~~~~~~~({\rm by  ~H\ddot{o}lder~inequality,~(3.3)~and ~the ~definition~ of ~} h ).
 \end{eqnarray*}
 For $\widetilde{JJ}_2$, we have
\begin{eqnarray*}
\widetilde{JJ}_2
&\leq &
C \|h\|_{L^1}^{\tilde{s}}
\sum_{i=1}^{\infty}
\int_{ Q^{x_0}_{2^{k_0+i+1}}\backslash Q^{x_0}_{2^{k_0+i}}}|b(x)-b_{Q^{x_0}_{2^{k_0}}}|^{\tilde{s}}|x-x_0|^{-n\tilde{s}}dx
\\
&&~~~~~~~~~~~~~~~~~~~~~~~~~~~~~~~~~~~({\rm by}~(1.2))
\\
&\leq &
C \|h\|_{L^1}^{\tilde{s}}
\sum_{i=1}^{\infty}
2^{-n\tilde{s}i}|Q^{x_0}_{2^{m_0}}|^{-\tilde{s}}
\int_{ Q^{x_0}_{2^{k_0+i+1}}}|b(x)-b_{Q^{x_0}_{2^{k_0}}}|^{\tilde{s}}dx
\\
&\leq &
C \|b\|_{BMO}^{\tilde{s}}
\|h\|_{L^{s}}^{\tilde{s}}|Q^{x_0}_{2^{m_0}}|^{\tilde{s}/s'}|Q^{x_0}_{2^{m_0}}|^{-\tilde{s}}|Q^{x_0}_{2^{m_0}}|
\sum_{i=1}^{\infty} 2^{-n(\tilde{s}-1)i}i^{\tilde{s}}
\\
&&
~~~~~~~~~~~~~~~~~~~~~~~({\rm by~H\ddot{o}lder~inequality~ and ~~(3.3)})
\\
&\leq &
C\|b\|_{BMO}^{\tilde{s}} |Q^{x_0}_{2^{m_0}}|w(Q^{x_0}_{2^{m_0}})^{-\tilde{s}/p}
\\
&&
~~~~~~~~~~~~~~~~({\rm by  ~the ~definition~ of ~} h ).
 \end{eqnarray*}
Combining the above estimates, we get (3.14).

 Next, let us prove (3.15) for a $(p,s,w)$-block $h$ with
 supp $h\subseteq Q = Q^{x_0}_{2^{m_0}}$. We need to estimate $A_{\tilde{s}}$,   as  $A_{{s}}$ in (3.11), we need to estimate $J_{\tilde{s}}$ and $JJ_{\tilde{s}}$.

\par For   $JJ_{\tilde{s}}$,  noticing that a $(p,s,w)$-block is also a $(p,\tilde{s},w)$-block for $\tilde{s}<s\leq \infty$,  by the estimate (3.13) for $JJ_{ s}$,
 we have
\begin{eqnarray*}
JJ_{\tilde{s}}
\leq
  C \|b\|_{BMO}^{\tilde{s}} \left(
|Q^{x_0}_{2^{k_0}}|^{-1/\tilde{s}}w(Q^{x_0}_{2^{k_0}})^{1/p}
\right)^{d\tilde{s}/(d-c)}
 \|h \|^{\tilde{s}}_{L^{\tilde{s}}   },
\end{eqnarray*}
and then, using H\"{o}lder inequality and the definition of $h$, we get
\begin{eqnarray*}
JJ_{\tilde{s}}
 \leq
  C \|b\|_{BMO}^{\tilde{s}}
  \left(
|Q^{x_0}_{2^{k_0}}|^{-1/\tilde{s}}w(Q^{x_0}_{2^{k_0}})^{1/p}
\right)^{d\tilde{s}/(d-c)}
|Q|w(Q)^{-\tilde{s}/p}.
\end{eqnarray*}

\par For $J_{\tilde{s}}$, by $\frac{rp}{r-1}<  \tilde{s}$, and
 $w\in RH_r$,  we then have
 \begin{equation}\label{5.2}
|Q^{x_0}_{|x-x_0|}|^{-1/\tilde{s}}w(Q^{x_0}_{|x-x_0|})^{1/p}
\leq
|Q^{x_0}_{R}|^{-1/\tilde{s}}w(Q^{x_0}_{R})^{1/p}
 \end{equation}
for $|x-x_0|\leq R$, (see also (6.4) in  \cite{L}).

 \par Thus,  using (3.16) (take $R=2^{k_0+1}$),  (3.14) and (1.1), we have
\begin{eqnarray*}
J_{\tilde{s}}
&=&
 \int _{Q^{x_0}_{2^{k_0+1}}}
|T_b h(x)|^{\tilde{s}} \left( |Q^{x_0}_{|x-x_0|}|^{-1/\tilde{s}}w(Q^{x_0}_{|x-x_0|})^{1/p}\right)^{d\tilde{s}/(d-c)} dx
\\
& \leq &
\left(
|Q^{x_0}_{2^{k_0+1}}|^{-1/\tilde{s}}w(Q^{x_0}_{2^{k_0+1}})^{1/p}
\right)^{d\tilde{s}/(d-c)}
 \|T_b h \|^{\tilde{s}}_{L^{\tilde{s}}   }
 \\
& \leq & C \|b\|_{BMO}^{\tilde{s}}
  \left(
|Q^{x_0}_{2^{k_0}}|^{-1/\tilde{s}}w(Q^{x_0}_{2^{k_0}})^{1/p}
\right)^{d\tilde{s}/(d-c)}
|Q|w(Q)^{-\tilde{s}/p}.
\end{eqnarray*}

Combining  the  estimate of $JJ_{\tilde{s}}$ with $J_{\tilde{s}}$, and using (1.1),
  we have
  \begin{eqnarray*}
A _{\tilde{s}}&\leq &
C\|b\|_{BMO}
\left(
|Q^{x_0}_{2^{m_0}}|^{-1/\tilde{s}}w(Q^{x_0}_{2^{m_0}})^{1/p}
\right)^{d/(d-c)}
|Q|^{1/\tilde{s}}w(Q)^{-1/p},
\end{eqnarray*}
 that is (3.15), noticing  $ Q= Q^{x_0}_{2^{m_0}} $.
 Thus, (3.9) is proved under the conditions of Theorem 2.7.
\end{proof}
 This concludes  the proof of Theorem 2.7.
\begin{proof}[Proof of Theorem 2.9]
Each  $T^{1}$  associated  with commutators $T_b^{1}$  obviously satisfy (2.7) (or see \cite {SW})), (1.2) follows.
  It is known that each  $T_b^{1}$ is bounded on  $L^s$ with $1<s<\infty$.
  Each  $T^{1}$ is obviously linear.

Then, the theorem follows from   Theorem 2.1(i) and 2.3(i).
\end{proof}

\begin{proof}[Proof of Theorem 2.11]
$M$ satisfies (1.2), (see \cite{L}). It is known that  $M_b$ is bounded on $L^s$ with $1<s< \infty$.

 \par Let $w\in A_{q,r}\cap P$  with $ 0<q <p  $  and  $ 1<r <\infty  $.
 Let $f(x)=\sum_{l}\lambda_l a_l\in  BH^{p,s}_w$, where $\{a_l\}$ is a sequence   $(p,s,w)$-blocks and $\{\lambda_l\}$ is a sequence  numbers with $ \sum_{l}|\lambda_l|^{\bar{p}} <\infty$.
By Fact 1.6, we see that $f(x)=\sum_{l}\lambda_l a_l$ converges   $w$-a.e., for  $  1 < s\leq \infty  $ and  $  rp/(r-1) \leq s$, by Lemma 3.4, it holds   a.e., it follows,
$|f(x)|\leq \sum_{l}|\lambda_l| |a_l(x)|,$ a.e..
Each $M_b a_l$ is well defined since $a_l\in L^s$ with $1< s<\infty$.  Then,
we have by   Minkowski inequality   that
 $M_b f(x)\leq \sum_{l}|\lambda_l| M_b a_l(x),$ a.e.,  by using Lemma 3.4 again,  it holds $w$-a.e.. That
is $M_b $ satisfies (1.3).

\par Then, the theorem follows from Theorem 2.1(ii) and Theorem 2.3(ii).
\end{proof}

\begin{proof}[Proof of Theorem 2.13  ]
 $\tilde{T}^{\Omega, *}$ obviously  satisfies (2.7), (1.2) follows. It is known that the commutator  $\tilde{T}_b^{\Omega, *}$ is bounded on $L^s$  for all $1<s<\infty$.

\par The  truncated operator  associated  with      $\tilde{T}^{\Omega, *}$  obviously  satisfy (2.7), (1.2) follows.
It  obviously is  linear and the commutator of this truncated operator is also  bounded on $L^s$  for all $1<s<\infty$. By Theorem 2.1(i), for this commutators (2.1) holds $w$-a.e., it follows that $\tilde{T}_b^{\Omega, *}$ satisfies (1.3).

 Then,   the result of (i) follows from  Theorem 2.1(ii), and the result of (ii) follows  from Theorem 2.3(ii).
\end{proof}

\begin{proof}[Proof of Theorem   2.15]
 Each $T^2$ obviously  satisfy (2.7), (1.2) follows. It is known that they are also bounded on $L^s$  for all $1<s<\infty$.

\par Each  truncated operator  associated  with     $ H^*,R^{j,*}, T^{\Omega, *},T^{F,*},T^{CZ,*}$ and $ T_{}^{P,*}$  obviously  satisfy (2.7)  and the constant $C$ in (2.7) is independent of the  truncated number $\varepsilon$, (1.2) follows. $C^{\xi} ,  B^{(n-1)/2,R }$ and $C^{P,d,1}$ (or $C^{P,d,n}$) satisfy (2.7) and the constant $C$ in (2.7) is independent of $\xi,R$ and $P$, (1.2) follows, where $C^{\xi} $ and $  B^{(n-1)/2,R }$  as before, and $ C^{P,d,n} f(x)=\int_{{\bf R}^n} e^{  i P(x- y)} k(x-y)  f(y)dy$.
These operators  obviously are  linear and are also  bounded on $L^s$  for all $1<s<\infty$. By Theorem 2.5(i), for the commutators of these operators, (2.1) holds $w$-a.e., it follows that each $T_b^2$ satisfies (1.3).

Then,  the result of (i) follows from  Theorem 2.5(ii), and the result of (ii) follows  from Theorem 2.7(ii).
\end{proof}

\begin{proof}[Proof of Theorem   2.18]
   The result of (i) follows from  Theorem 2.5(i), and the result of (ii) follows from Theorem 2.7(i).
\end{proof}

\bibliographystyle{amsplain}

\end{document}